\documentclass[11pt]{amsart}
\usepackage{fullpage,cite,hyperref,enumitem}
\usepackage{amssymb,amsmath,amsthm,mathtools,extpfeil}
\usepackage[all,cmtip]{xy}
\usepackage{rotating,graphicx}
\usepackage{tikz}

\newtheorem{thm}[equation]{Theorem}
\newtheorem*{thm*}{Theorem}
\newtheorem{thmA}{Theorem}

\newtheorem{lem}[equation]{Lemma}
\newtheorem{prop}[equation]{Proposition}

\newtheorem{conj}{Conjecture}

\newtheorem{ques}[equation]{Question}
\theoremstyle{definition}

\newtheorem*{rmks}{Remarks}
\newtheorem*{open}{Open problem}
\newtheorem{ex}[equation]{Example}

\numberwithin{equation}{section}

\DeclareMathOperator{\id}{id}
\DeclareMathOperator{\Hom}{Hom}
\DeclareMathOperator{\Aut}{Aut}
\DeclareMathOperator{\Lip}{Lip}
\DeclareMathOperator{\Map}{Map}

\DeclareMathOperator{\lengt}{length}
\DeclareMathOperator{\vol}{vol}
\DeclareMathOperator{\img}{im}
\DeclareMathOperator{\st}{st}
\DeclareMathOperator{\lk}{lk}
\DeclareMathOperator{\rk}{rk}
\DeclareMathOperator{\mass}{mass}

\DeclareMathOperator{\Dil}{Dil}
\newcommand{\ph}{\varphi}
\newcommand{\epsi}{\varepsilon}

\begin{document}
\title{Plato's cave and differential forms}
\author{Fedor Manin}
\address{Department of Mathematics, Ohio State University, Columbus, Ohio, USA}
\email{manin@math.toronto.edu}
\begin{abstract}
  In the 1970s and again in the 1990s, Gromov gave a number of theorems and
  conjectures motivated by the notion that the real homotopy theory of compact
  manifolds and simplicial complexes influences the geometry of maps between
  them.  The main technical result of this paper supports this intuition: we show
  that maps of differential algebras are closely shadowed, in a technical sense,
  by maps between the corresponding spaces.  As a concrete application, we prove
  the following conjecture of Gromov: if $X$ and $Y$ are finite complexes with
  $Y$ simply connected, then there are constants $C(X,Y)$ and $p(X,Y)$ such that
  any two homotopic $L$-Lipschitz maps have a $C(L+1)^p$-Lipschitz homotopy (and
  if one of the maps is constant, $p$ can be taken to be $2$.)  We hope that it
  will lead more generally to a better understanding of the space of maps from
  $X$ to $Y$ in this setting.
\end{abstract}
\maketitle

\section{Introduction}

In 1996, Princeton University invited several prominent mathematicians, including
Misha Gromov, to give a series of lectures entitled ``Prospects in Mathematics'';
each speaker would discuss their views on future directions in their field.
Gromov's talk was entitled ``Quantitative homotopy theory''\footnote{Actually,
  due to a mistake somewhere along the way, the title of the talk was given as
  ``Qualitative homotopy theory''---exactly the opposite of what Gromov meant!
  The notes were published in the proceedings of the conference as \cite{GrQHT}.}
and advanced the idea that the central questions of algebraic topology---such as,
are two maps homotopic?---should be refined by asking about the sizes of the
objects produced.

Indeed, a central weakness in the extremely powerful results of algebraic and
geometric topology is their indirectness: they are obtained by reducing geometric
problems to those of homotopy theory and homotopy problems to algebra, leaving us
very little understanding of the shapes of the solutions, or whether finding the
correlates of these solutions back in the geometric world is easy or hard.  A
beautiful example of this phenomenon is the result of Nabutovsky \cite{NabSph}
that, although every embedded codimension-one sphere in the $n$-disk ($n>4$) is
isotopic to the boundary due to Smale's $h$-cobordism theorem, the complexity of
such an isotopy (measured by the size of an embedded normal bundle) cannot be
bounded by any computable function of the complexity of the original embedded
sphere.

Nabutovsky's result follows from the unsolvability of the triviality problem for
groups: while the embedded sphere is simply connected, it has no way of knowing
that it is (i.e.~no algorithm can distinguish it from a homology sphere with
nontrivial $\pi_1$) and in particular it cannot know that it is isotopic to the
boundary.  With objects that are aware of their simple connectivity, homotopy
theory tends to be more computationally tractable\footnote{But this does not mean
  tractable in an absolute sense; see \cite{CWetal} and \cite{FiVo} for two
  contrasting perspectives.}, and the solutions are correspondingly less complex.
This seems to stem from the fact that the algebra describing them is commutative.

Indeed, in the setting of this paper, geometric complexity is controlled quite
closely by the algebraic structure of maps.  One of the most important tools for
studying simply connected spaces is rational homotopy theory, first developed by
Quillen and Sullivan in the 1970s.  Almost immediately, Gromov realized that
Sullivan's machine has geometric consequences, providing lower bounds on the
complexity of maps in a given homotopy class, and conjectured that these lower
bounds are sharp.  In the intervening years these ideas have been developed
further by Gromov, and more recently by Weinberger, Ferry, Chambers, Dotterrer,
Guth, and the author.

This paper seeks to strengthen the link between rational homotopy and geometry by
showing a kind of inverse result to the reduction to algebra: that, at least in
the world of compact spaces, the algebraic maps produced by the theory always
have reasonably close geometric doppelgangers.  This is robust enough to provide
an almost immediate proof of one of the conjectures in Gromov's talk.  We also
provide applications which nibble at the margins of some other problems; it is
the author's hope that with more effort and new techniques, other broader
applications will be found.

Our results do not directly resolve any metric problems beyond homotopy theory;
in settings such as cobordism or embedding, one also needs to analyze the
reduction to homotopy theory.  However, again the hope is that these results will
simplify that task to some degree.

\subsection{Gromov's conjectures}

Gromov's 1990s questions have roots stretching back to the 1978 paper
\cite{GrHED}, where he showed the following result:
\begin{thm*}
  Let $Y$ be a simply connected finite complex with a (reasonable) metric.  Then
  $\pi_n(Y)$ has \emph{polynomial growth}, i.e.~the number of elements which have
  a representative with Lipschitz constant at most $L$ is bounded by a polynomial
  in $L$. 
\end{thm*}
There are several reasons why the Lipschitz constant is a natural notion of
complexity here.  First, if two such spaces are homotopy equivalent, then they
are Lipschitz homotopy equivalent; this means that Lipschitz invariants such as
the asymptotics of the growth function of $\pi_n(Y)$ are as natural as homotopy
invariants in this context.  Conversely, every $L$-Lipschitz map from $S^n$ is
homotopic to a simplical map on a subdivision with $\sim L^n$ simplices; see
Prop.~\ref{prop:qSAT}.  This means that the number of bits of information needed
to (homotopically) specify an $L$-Lipschitz map is $\Theta(L^n)$.  In this
framing, Gromov's result says that for simply connected targets, this is a
significant overestimate.\footnote{On the other hand, it is sharp for examples
  like $\pi_n(T^n \vee S^n)$: there are $\Theta(\exp(L^n))$ homotopy classes of
  maps with Lipschitz constant $L$.}  On a more geometric level, the Lipschitz
constant bounds the sizes of pullbacks of forms; this is its main property used
in the proof.

Gromov's theorem vastly generalizes an observation about Hopf invariants.
Suppose that $f:S^3 \to S^2$ is a smooth $L$-Lipschitz map between round unit
spheres.  Denote the volume form on $S^2$ by $d\vol$; $f^*d\vol$ is a closed
2-form in $S^3$ and therefore $f^*d\vol=d\alpha$ for some 1-form $\alpha$.  Then,
following J.H.C.~Whitehead, the Hopf invariant of $f$ is given by
$$H(f)=\int_{S^3} \alpha \wedge f^*d\vol.$$
Now we look at the $L^\infty$-norms of these forms (i.e.\ the supremum of their
values taken over all frames of unit vectors.)  We know
$\lVert f^*d\vol\rVert_\infty \leq L^2$ and we can choose $\alpha$ so that
$\lVert\alpha\rVert_\infty \leq C\lVert f^*d\vol \rVert_\infty$.\footnote{The
  latter inequality is not entirely obvious and Gromov did not prove it in this
  paper, although he did give a sketch in \cite{GrMS}; we prove it in Section 2.}
Therefore we have the inequality
$$H(f) \leq C\vol(S^3)L^4.$$
Indeed, up to a constant, this is sharp: a map with Lipschitz constant $O(L)$ and
Hopf invariant $L^4$ can be built via the composition
$$S^3 \xrightarrow{\text{Hopf map}} S^2 \xrightarrow{\text{degree }L^2} S^2.$$
To Gromov, this and other examples suggested the following conjecture:
\begin{conj}[{\cite{GrQHT}}] \label{conj:growth}
  The estimate on the growth of $\pi_n(Y)$ provided by the method of \cite{GrHED}
  is sharp.
\end{conj}
\noindent Thus far, we have failed to give an algorithmic description of this
method and thus this conjecture remains an impressionistic, ill-defined one.  In
particular, in \S\ref{S:pi_k} we give an example where a candidate algorithm
based on the work of Sullivan fails to produce the correct bound.  Nevertheless,
the examples in that section illustrate the intuition that suggests that an
algorithm can be found.

Gromov returned to this theme in the 1990s in \cite[Ch.~7]{GrMS} and the
conference paper \cite{GrQHT}.  In these works he presented two other conjectures
which are relevant to the present work.  The first concerns a cousin of the
growth of homotopy groups.  Given an element $\alpha \in \pi_n(Y)$, define the
\emph{distortion function}
$$\delta_\alpha(k)=\inf\{\Lip f \mid f:S^n \to Y,[f]=k\alpha \}.$$
Note that in all cases $\delta_\alpha(k)=O(k^{1/n})$; this is because one can
always find a representative of $k\alpha$ by precomposing a representative of
$\alpha$ with a degree $k$ map $f_k:S^n \to S^n$ with $\Lip f_k=O(k^{1/n})$.  On
the other hand, for some $\alpha$ one can do better; we say such $\alpha$ are
\emph{distorted}, whereas those for which $\delta_\alpha(k)=\Theta(k^{1/n})$ are
\emph{undistorted}.  In this language, what Gromov showed in \cite{GrHED} is that
the generator of $\pi_{2n+1}(S^n)$ is distorted, and also that when $Y$ is simply
connected, $\delta_\alpha$ for an element $\alpha \in \pi_n(Y)$ is always
$\Omega(k^{1/(2n)})$.
\begin{conj}[\!\!\!{\cite{GrQHT}}] \label{conj:dist}
  When $Y$ is simply connected, an element $\alpha \in \pi_n(Y)$ is undistorted
  if and only if it has nonzero image under the Hurewicz map to
  $H_n(Y;\mathbb{Q})$.  If it is distorted, then $\delta_\alpha(k)=O(k^{1/n+1})$.
\end{conj}
\noindent There is a ``strong'' but again impressionistic version of the
conjecture which states that the bound implied by \cite{GrHED} is sharp, and
which is equivalent to Conjecture \ref{conj:growth}.

One may try to formulate similar conjectures more generally for the set of
mapping classes $[X,Y]$ where $X$ is not necessarily a sphere.  In \cite{GrMS},
Gromov suggests that the growth of $[X,Y]$ should be asymptotic to $L^\alpha$ for
some integer $\alpha$ determined by the minimal models of $X$ and $Y$.  This is
disproved in the companion paper \cite{IRMC} by the author and Weinberger;
however, we do show there that the growth of $[X,Y]$ is at least bounded above by
a polynomial.

Since the integral homotopy classes can be thought of in general as the integer
points of an algebraic variety\footnote{This variety is cut out of the space of
  graded algebra maps between free DGAs by equations forcing it to be a chain
  map.}, one cannot say much in the way of lower bounds on growth.  Perhaps
results can be obtained when this variety has particularly nice properties, or in
instances where there is more structure, for example for $\Aut(Y)$, which
Sullivan \cite{SulLong} demonstrated is an arithmetic subgroup of an algebraic
group of rational automorphisms.

One could formulate a weaker conjecture, somewhat analogous to Conjecture
\ref{conj:dist}.  While the notion of distortion only makes sense when the set of
mapping classes $[X,Y]$ is a group, Gromov sketches an argument in \cite{GrQHT}
that the Lipschitz constant of a map gives an upper bound on its
obstruction-theoretic rational homotopy invariants.  One can then guess that
every class which can be realized via small enough such invariants has a
$CL$-Lipschitz representative.  This guess also turns out to be false in general,
as will be explained in a forthcoming paper.

Nonetheless, a relative analogue can be stated in this more general setting.
\begin{conj}[{\cite{GrQHT}}] \label{conj:htpy}
  Let $f \simeq g:X \to Y$ be $L$-Lipschitz maps from a finite complex to a
  finite simply connected complex.  Then there is a polynomially bounded function
  $P_{X,Y}(L)=O(L^{p(X,Y)})$ such that there is a homotopy between $f$ and $g$
  through $P_{X,Y}(L)$-Lipschitz maps.
\end{conj}
\noindent Gromov remarked that he could not think of an example where this
polynomial had to be nonlinear.\footnote{On the other hand, if we allow
  $\pi_1(Y)$ to be nontrivial and take $X=S^1$, this corresponds to the so-called
  isodiametric function of $\pi_1(Y)$, which for certain groups grows faster than
  any computable function \cite{Ger}.}

In the past 20 years, there has been some incremental progress on these
conjectures.  An unpublished result of Shmuel Weinberger (which appears in the
author's PhD thesis \cite{thesis}) shows the following weak version of the
distortion conjecture: there are no rationally nontrivial distorted elements in
$\pi_*(Y)$ if and only if the Hurewicz map
$$\pi_*(Y) \otimes \mathbb{Q} \to H_*(Y;\mathbb{Q})$$
is injective.  The proof uses the fact that distortion is well-understood for
generalized Whitehead products.  Conjecture \ref{conj:htpy} is proven for target
spaces $Y$ whose rational homotopy structure is relatively simple (including
spheres, H-spaces, and homogeneous spaces of Lie groups) in the series of papers
\cite{FWPNAS}, \cite{CDMW}, and \cite{CMW}.

In this paper, we prove results about Lipschitz homotopies which generalize those
of \cite{CDMW} and \cite{CMW} and are actually somewhat stronger than Conjecture
\ref{conj:htpy}.  Define the \emph{length} of a homotopy (sometimes also referred
to as \emph{width}) to be the maximal Lipschitz constant of its restrictions to
$\{x\} \times [0,1]$, and its \emph{thickness} to be the maximal Lipschitz
constant of its restrictions to $X \times \{t\}$.  Gromov's conjecture only asks
about thickness; here is a summary of the results of \S\ref{S:liphom}:
\begin{thmA} \label{summary:htpy}
  Let $Y$ be a finite simply connected complex and $X$ a finite complex of
  dimension $n$.
  \begin{enumerate}[label={(\roman*)}]
  \item There are constants $C(X,Y)$ and $p(X,Y)$ such that any homotopic
    $L$-Lipschitz maps $f \simeq g:X \to Y$ are homotopic via a homotopy of
    length $C$ and thickness $C(L+1)^p$.
  \item Moreover, any nullhomotopic $L$-Lipschitz map is nullhomotopic via a
    homotopy of length $C$ and thickness $C(L+1)^2$.
  \item If in addition $Y$ has positive weights (an algebraic condition on the
    rational homotopy structure), then any nullhomotopic $L$-Lipschitz map is
    nullhomotopic via a homotopy of linear thickness and length $C(L+1)^{n-1}$.
  \end{enumerate}
\end{thmA}
The latter two bounds are sharp: there are spaces for which one parameter cannot
be decreased without increasing the other.  On the other hand, it's not clear
whether linear thickness is achievable for some classes of maps not satisfying
(iii).

The growth and distortion conjectures are more resistant for reasons which are
explained later in the introduction, but we do prove a set of results for
symmetric spaces:
\begin{thmA} \label{summary:sym}
  Let $Y$ be a simply connected finite complex which has the rational homotopy
  type of a Riemannian symmetric space.  Write
  $\eta_k:\pi_k(Y) \to H_k(Y;\mathbb{Q})$ for the Hurewicz homomorphism.
  \begin{enumerate}[label={(\roman*)}]
  \item The distortion of an element $\alpha \in \pi_n(Y)$ is $\Theta(k^{1/n})$ if
    $\eta_k(\alpha) \neq 0$ and $\Theta(k^{1/(n+1)})$ otherwise.  (This proves the
    ``strong'' distortion conjecture for such spaces.)
  \item The size of the $L$-ball in $\pi_n(Y)$ is
    $\Theta(L^{n\rk\img\eta_k+(n+1)\rk\ker\eta_k})$.
  \item Nullhomotopic $L$-Lipschitz maps $X \to Y$, for any finite complex $X$,
    have nullhomotopies whose Lipschitz constant is slightly superlinear in $L$.
  \end{enumerate}
\end{thmA}
I believe that the sharp bound on sizes of nullhomotopies in this case is linear,
but (iii) is an improvement over Theorem \ref{summary:htpy} which only gives a
quadratic bound.

\subsection{Minimal models and DGA maps}

To state more precisely the technical ideas in this paper, we must delve into
Sullivan's model of rational homotopy theory.  This is discussed in greater
detail in \S3 and we also refer the reader to \cite{SulLong} and \cite{GrMo} for
detailed exposition.  More accurately, what we give here is \emph{real} homotopy
theory; the results are less impressive than those of rational homotopy theory in
some respects that are irrelevant to the ideas in this paper, but this theory has
the advantage of working with off-the-shelf differential forms which behave
nicely with respect to smooth maps.

For our purposes, the main points of Sullivan's theory are that the algebra of
smooth differential forms $\Omega^*Y$ on a compact manifold $Y$ with boundary is
a fairly good homotopy-theoretic model for the space $Y$ itself; and that it in
turn is modeled by a much smaller, easily described algebra closely related to
the Postnikov tower of $Y$.

More precisely, we think of these as differential graded algebras (DGAs), that
is, chain complexes (in this case over $\mathbb{R}$) equipped with a
multiplication which satisfies the graded Leibniz rule.  If $Y$ is simply
connected, then there is a homotopy equivalence (under a well-known notion of
homotopy of DGAs which we define in \S3) $m_Y:\mathcal{M}_Y^* \to \Omega^*Y$
where $\mathcal{M}_Y^*$ is a DGA of finite type (i.e.~generated by a finite
vector space in every degree.)  This \emph{minimal model} has a number of nice
properties, but all that matters for us is that given a map $f:X \to Y$ from some
manifold $X$, we can describe the homomorphism
$f^*m_Y:\mathcal{M}^*_Y \to \Omega^*X$ using a finite number of
differential-form-valued invariants.  Indeed, up to homotopy, this description
can be made finitary in a much stronger sense.

Write $[\mathcal{M}^*_Y,\Omega^*X]$ for the set of homotopy classes of DGA
homomorphisms.  Then $f \mapsto f^*m_Y$ induces a well-defined map
$[X,Y] \to [\mathcal{M}^*_Y,\Omega^*X]$ which is finite-to-one by
\cite[Thm.~10.2(i)]{SulLong}.  Moreover, in various cases where these sets have a
group structure, this map is actually the homomorphism ${} \otimes \mathbb{R}$.

\subsection{Existence of shadows}

In this paper, we study the algebraicization map $f \mapsto f^*m_Y$ more closely,
as a continuous map
$$\mathbf{Alg}:\Map(X,Y) \to \Hom(\mathcal{M}_Y^*,\Omega^*X),$$
where the latter object is equipped with a metric induced by its homotopy theory.
We can think of homomorphisms $\mathcal{M}_Y^* \to \Omega^*X$ as ``platonic
forms'' of maps.  These include, of course, the pullbacks of genuine maps
$X \to Y$, just as a committed Platonist would have to admit that the world of
concepts includes the concept of any particular object in the real world, as well
as abstractions at various levels.  But most platonic forms are indeed abstract.
Moreover, $\mathbf{Alg}$ is far from being a homotopy equivalence, even on
connected components, since many algebraic homotopies have non-integer and even
irrational invariants.

Nevertheless, the main technical theorem of this paper is that we can produce
``almost inverse images'' under $\mathbf{Alg}$.  Suppose $Y$ is compact and $X$
has bounded geometry.  If $\ph:\mathcal{M}_Y^* \to \Omega^*X$ is in the connected
component of a genuine map, then it has a shadow $f:X \to Y$ in the Plato's cave
of genuine maps such that $f^*m_Y$ is reasonably close to $\ph$, as measured by
the size of an (algebraic) homotopy between them.  Moreover, the Lipschitz
constant of $f$ is closely related to a natural geometric functional on $\ph$
which we call the \emph{formal dilatation}.  Most of our applications actually
use the relative form of this statement:
\begin{thm*}[Shadowing principle, informal version]
  Let $A \subset X$ be a subcomplex and $u:A \to Y$ an $L$-Lipschitz map.  Then
  any extension $\ph:\mathcal{M}_Y^* \to \Omega^*X$ over $X$ of $u^*m_Y$ which is
  in the relative homotopy class of a genuine extension $\tilde u:X \to Y$ of $u$
  has a nearby shadow $f:X \to Y$ which is in the same relative homotopy class
  and has Lipschitz constant at most $CM+C$, where $M$ is the formal dilatation
  of $\ph$.
\end{thm*}
The precise statement is given in Theorem \ref{thm:main}.

The significance of this is that platonic maps are sometimes easier to construct
than genuine maps, since they have fewer moving parts; this makes it easier to
construct new geometrically bounded objects.  For example, it is much easier to
produce a homotopy in the algebraic world than the geometric one, and this is
what gives us our powerful results about homotopies.  Other new results are
obtained by harnessing scaling automorphisms of DGAs.  On the other hand, in
order to realize the full potential of the shadowing principle, we need
additional techniques for constructing DGA homomorphisms.

\subsection{The method of Guth}

The proof of the shadowing principle is inspired by Larry Guth's recent
streamlined proof \cite{Guth} of the main homotopical result of \cite{CDMW}.  We
give an outline of this proof here.

Suppose we have a nullhomotopic $L$-Lipschitz map $f:S^m \to S^n$, where either
$n$ is odd or $m<2n-2$; we would like to construct a $C(m,n)L$-Lipschitz
nullhomotopy $F:S^m \times [0,1] \to S^n$.  First of all, we may assume, by a
quantitative simplicial approximation result, that $f$ is the composition of a
simplicial map from some triangulation of $S^m$ at scale
$\sim 1/L$\footnote{I.e.\ with simplices uniformly bilipschitz to a linear
  simplex with edgelength $1/L$.} to $\partial\Delta^{n+1}$ and a smooth map that
contracts all but one of the faces of $\partial\Delta^{n+1}$.  Next, we choose
some uncontrolled nullhomotopy $F$ of this map.  We will homotope this to a
controlled homotopy.

First, choose a triangulation $X$ of $S^m \times [0,1]$ also at scale $\sim 1/L$,
restricting to our triangulation of $S^m$ at $t=0$.  We will proceed by induction
on the skeleta of this triangulation.  The key point is that at the $k$th step we
will make sure that the $k$-simplices of $X$ are mapped to $S^n$ in one of a
fixed set of ways, depending only on $m$ and $n$.  Then the Lipschitz constant is
bounded by
$$\sim (\text{max Lipschitz constant of a restriction to a simplex})\cdot
(\text{min edge length of }X)^{-1}.$$

For $k<n$, we do this simply by sending the whole $k$-skeleton to the basepoint
of $S^n$.  This may make the homotopy even worse than it was on higher simplices,
but we will fix this in future steps.  This gives us a homotopy $F_{n-1}$ which
sends $X^{(n-1)}$ to a point; if $m<n$, we are done.

The $n$th step is the trickiest, and it is here that we use some algebra.  Note
that since $F_{n-1}|_{X^{(n-1)}}$ is constant, $F_{n-1}$ has a well-defined degree on
$n$-simplices.  Let $c \in C^n(X)$ be the cochain whose value on simplices is
this degree.  Since $F_{n-1}$ is defined on $(n+1)$-cells, this is a cocycle.

We compare this to another cocycle that describes the ``ideal'' behavior of such
a nullhomotopy.  The piecewise smooth form $f^*d\vol \in \Omega^n(S^m)$ is exact
since $f$ is nullhomotopic.  Moreover, $\lVert f^*d\vol \rVert_\infty \leq L^n$;
by an isoperimetric result for forms, reproven in this paper as Lemma
\ref{lem:IP}, we can find an $\alpha \in \Omega^{n-1}(S^m)$ such that
$d\alpha=f^*d\vol$ and $\lVert\alpha\rVert_\infty \leq C(m,n)L^n$.  Let
$\pi:X \to S^m$ be the obvious projection; then we define a cocycle
$w \in C^n(X;\mathbb{R})$ by sending each $n$-simplex $p$ to
$$w(p)=\int_p \left((1-t)\pi^*f^*d\vol+(-1)^n\pi^*\alpha \wedge dt\right).$$
The $L^\infty$ bound then implies that $\lvert w(p) \rvert \leq 1+C(m,n)$.

Note that $w=c$ on the simplices of $S^m \times \{0,1\}$.  Thus
$w-c \in C^n(X,S^m \times \{0,1\};\mathbb{R})$ is a relative cocycle and hence
a relative coboundary since $m \geq n$, $w-c=\delta b$ for some
$b \in C^{n-1}(X,S^m \times \{0,1\};\mathbb{R})$.  Now we homotope $F_{n-1}$ to a
map $F_n$ as follows.  The homotopy will be constant on $X^{(n-2)}$.  On each
$(n-1)$-simplex $q$, we make the homotopy trace out a map of degree $[b(q)]$,
i.e.~the nearest integer to $b(q)$, and return to the constant map to the
basepoint.  This then fixes the degree of $F_n$ on each $n$-simplex $p$; this
degree within distance $(n+1)/2$ from $(c+\delta b)(p)=w(p)$.  This is bounded by
a constant depending only on $m$ and $n$; for each degree below this bound, we
fix a specific map on $\Delta^n$ and homotope to that map.

Now let $k>n$; by induction, we have a map $F_{k-1}$ which takes a finite set of
values on $(k-1)$-simplices.  In particular, there is a finite set of values that
it can take on the boundary of any $k$-simplex $p$.  Moreover, given
$F_{k-1}|_{\partial p}$, the possible relative homotopy classes of $F_{k-1}|_p$ form
a torsor for $\pi_k(S^n)$, which is finite by assumption.  Thus we can fix a map
in each such relative homotopy class and homotope to an $F_k$ whose restriction
to $p$ is that map.  Once $k=m+1$, we have completed the proof.

Let us return now to the $n$th step.  In this paper, we reinterpret this as
follows.  The form
$$(1-t)\pi^*f^*d\vol+(-1)^n\pi^*\alpha \wedge dt$$
should be thought of as an algebraic nullhomotopy of the form $f^*d\vol$ which
describes $f$ up to finite uncertainty; this is made precise in \S3.  We
construct our controlled nullhomotopy by pulling the uncontrolled homotopy $F$
as close as we can to the controlled, but purely algebraic one.

In more general situations, the map and its nullhomotopy cannot be fully
described by a single form.  Instead, the description of a map $X \to Y$ is an
algebra homomorphism $\mathcal{M}_Y^* \to \Omega^*X$.  However, we can still use
roughly the same procedure: take an uncontrolled geometric homotopy $F$ and a
controlled algebraic one $\Phi$; as long as they are homotopic to each other in
the algebraic sense, we can gradually pull $F$ towards $\Phi$, skeleton by
skeleton, until we get a geometric homotopy which is close to $\Phi$, and
therefore controlled.  This works not only for homotopies but for maps in a
relative homotopy class in general.

\subsection{Seeing outside the cave} \label{S:outside}

The method outlined in the previous section has an important weakness: in order
to get the bound we want, we need to first find a DGA homomorphism that satisfies
it.  In the case of homotopies, there is an algorithm described in \S3 which
constructs such a homomorphism.  The bound obtained this way, however, while
sharp in some instances, is not, for example, sharp in the case of maps
$S^3 \to S^2$.  Here the algebraic method yields a quadratic bound, whereas I
strongly suspect that the true bound is linear.  In fact, we produce an only
slightly superlinear bound in Theorem \ref{thm:weird} using a somewhat mysterious
ad hoc method.

Similarly, for elements of $\pi_n(Y)$ we can always produce not-too-large
representatives algorithmically, but if we use the most general construction such
representatives will not say anything nontrivial about distortion.

To highlight some of the uncertainties, we come back to maps $f:S^3 \to S^2$.  To
construct an algebraic nullhomotopy of such a map, it is enough to find a 1-form
$\alpha \in \Omega^1(S^3)$ with $d\alpha=f^*d\vol$ and $\eta \in \Omega^2(S^3)$
with $d\eta=\alpha \wedge f^*d\vol$.  By the aforementioned isoperimetric result,
we can find $\eta$ with $\lVert\eta\rVert_\infty \lesssim (\Lip f)^4$.  A quick
argument (provided by the anonymous referee and explained in \S\ref{S:spheres})
shows that this bound cannot in general be improved by choosing the forms in a
more clever way.  At the same time, Sasha Berdnikov \cite{Berd} has shown that
linear homotopies can always be constructed in this setting.  Thus the obvious
method of constructing algebraic homotopies cannot provide a sharp geometric
bound.

Of course, the problem does reduce to a question about whether there are
homomorphisms $\mathcal{M}_{S^2}^* \to \Omega(S^3 \times [0,1])$ with certain
$L^\infty$ bounds on the images of the generators.  The point is that the
existence of such homomorphisms seems potentially just as hard to decide as the
original questions about maps and homotopies.  The same sort of questions bedevil
any attempts at resolving Conjectures \ref{conj:growth} and \ref{conj:dist}
purely through DGA methods; all the proofs we have use some kind of self-maps
that allow us to use one representative to generate a whole class of maps,
whether geometrically or algebraically.

\subsection{Extensions and generalizations}

The shadowing principle has the advantage of being completely local.  Therefore a
number of extensions which are not shown in this paper nevertheless seem
achievable.  The author would like to thank David Kazhdan, Shmuel Weinberger, and
Tali Kaufman for raising some of these points.
\begin{enumerate}[leftmargin=*]
\item The results should hold for nilpotent targets as well as simply connected
  ones.  This requires more complicated induction procedures and perhaps some
  stipulations regarding basepoints.
\item The results should hold for various extensions of rational homotopy theory,
  once one has a good understanding of the relevant algebra.  This includes
  equivariant rational homotopy theory (see \cite{Scull}) and perhaps the
  rational homotopy theory of more general diagrams of spaces \`a la \cite{DF}
  (although this has never been explicitly developed) as well as sections of a
  fibration, or more generally for rational homotopy theory of maps fibered over
  some fixed base space.
\item The theorem holds for the case where the domain is an infinite complex of
  bounded geometry (although we do not give any applications that use this.)  In
  such complexes, one could have DGA homomorphisms which are not bounded, but are
  controlled within an $r$-ball around some basepoint by some function $f(r)$.
  Then by rescaling or varying the sizes of subdivisions, we can get an honest
  map with similar control on the Lipschitz constant.
\item One interpretation of the shadowing principle is that in some sense, the
  map $\Map(X,Y) \to \Hom(\mathcal{M}^*_Y,\Omega^*X)$ induced by pullback of the
  minimal model is ``almost dense'' and induces a near-equivalence between the
  Lipschitz constant on $\Map(X,Y)$ and a similar geometric functional on the
  other space.  One could ask whether there is a stronger notion of connectivity
  between the Morse landscapes of these functionals; this needs to be done with
  some care since the map is not a homotopy equivalence.  The $\pi_0$ version of
  this question is this: given a path in $\Hom(\mathcal{M}^*_Y,\Omega^*X)$
  between two genuine maps which is in the relative homotopy class of a genuine
  homotopy, can we find a genuine homotopy with similar geometry to the path?
  The $\pi_n$ questions can be formed similarly.  It seems that the answer must
  be yes, but to confirm this one needs to understand paths in the space of
  homomorphisms, most of which are not algebraic homotopies in the sense we use.

  We can find a closer topological equivalence by restricting to homomorphisms of
  polynomial forms with rational coefficients.  These are closely related to maps
  from $X$ to the rationalization of $Y$, as discussed in \cite{BrSz}.  However,
  these do not usually come from pullbacks of maps, so we would still have to use
  the space $\Hom(\mathcal{M}^*_Y,\Omega^*X)$ as a common refinement.  Since
  smooth forms seem closely approximable by polynomials, it is likely that the
  geometry of this space is likewise quite similar.
\end{enumerate}

\subsection{Applications to geometric problems}

One of the main motivations for studying quantitative algebraic topology is to
try to understand the solutions to problems in geometric topology.  The long-time
method of doing geometric topology is to reduce it to problems in homotopy
theory, then solve those problems using algebraic methods.  One could therefore
attempt to understand the solutions by putting geometric bounds on both the
reduction and the homotopy theory.  Here are some examples where this has been
achieved.

\begin{enumerate}[leftmargin=*]
\item In \cite{CDMW}, we gave a bound on the size of a nullcobordism of a
  nullcobordant manifold.  This is a quantitative version of Thom's cobordism
  theorem.  Here, the algebraic problem was a special case of Conjecture
  \ref{conj:htpy}; the geometric problem was to get a bound on the size of Thom's
  construction.
\item Already in \cite{GrHED}, Gromov uses his estimate on the growth of homotopy
  classes to bound the growth of embedding spaces.  By a theorem of Haefliger
  \cite{HaeBki}, when $2n>3(m+1)$, isotopy classes of embeddings of an
  $m$-manifold $M$ in $\mathbb{R}^n$ are in bijection with
  $\mathbb{Z}/2\mathbb{Z}$-equivariant homotopy classes of maps
  $(M \times M) \setminus \Delta \to S^{n-1}$.  One direction is easy: every
  embedding $f:M \to \mathbb{R}^n$ is sent to the map
  $$(m,n) \mapsto \frac{f(m)-f(n)}{\lvert f(m)-f(n) \rvert}.$$
  After forgetting a tubular neighborhood of the diagonal, this correspondence
  sends $L$-bilipschitz maps to $O(L^2)$-Lipschitz ones.  Using (the free
  $\mathbb{Z}/2\mathbb{Z}$-equivariant version of) Gromov's polynomial estimate
  and this explicit procedure, one sees that the number of homotopy classes of
  $L$-bilipschitz embeddings of $M$ in $\mathbb{R}^n$ is at most polynomial in
  $L$.
\end{enumerate}
It would be interesting to investigate the space of such embeddings in greater
detail.  The methods of this paper provide solutions to some of the requisite
algebraic problems.  However, translating this into embedding theory requires a
deeper, more geometric understanding of the correspondence going from equivariant
maps back to embeddings.
\begin{open}
  (1) Find a sharp estimate of the number of embeddings of some $M$ in
  $\mathbb{R}^n$ with a given bilipschitz constant or other geometric bound.
  (2) Find a bound on the difficulty of isotoping two isotopic embeddings (again,
  in terms of the bilipschitz constant or some other geometric bound.)
\end{open}
We hope that our results will induce more work on the geometric side of these
problems and many others.

\subsection{Structure of the paper}

Section 2 introduces some technical results about the geometry of simplicial
complexes which underpin the various proofs.  In Section 3, we discuss rational
homotopy theory in detail, including the geometric estimates introduced by
Gromov.  In Section 4, we state and prove the main technical result.
Applications, including the proofs of Theorems \ref{summary:htpy} and
\ref{summary:sym}, are discussed in the last section.

\subsection{Acknowledgements}

I would like to first of all thank Shmuel Weinberger, who introduced me to these
ideas and was present for every stage of their development, and without whom none
of this would have been possible.  The paper is part of a project that has been
going on for several years; our collaborators Steve Ferry, Greg Chambers,
and Dominic Dotterrer also contributed many of the ideas which pervade this work.

I would also like to thank the Israel Institute for Advanced Studies, where this
paper was conceived and largely written.  Without the IIAS, I could not have
benefitted from the wisdom of David Kazhdan and Tali Kaufman, who suggested
various generalizations of the main result and directions for further study.  I
am grateful to Sasha Berdnikov for pointing out an important error in an earlier
version, and to the anonymous referee for an extremely thorough reading leading
to numerous corrections and suggestions, from sign errors to major issues.
Finally, this paper owes its existence to Larry Guth and his incredibly
thoughtful and clever reinterpretation of our earlier ideas.  I cannot thank him
enough.

\section{Geometric preliminaries}

\subsection{Simplicial approximation}

A key principle in this paper, carried over from \cite{CDMW} and \cite{CMW}, is
local standardization of maps.  The simplest kind of such standardization is
simplicial approximation.  It was shown in \cite{CDMW} that on the right sort of
subdivision, simplicial approximation can be performed without increasing
Lipschitz constants too much.

Let $X$ be a simplicial complex with the standard metric.  We say a subdivision
of $X$ is \emph{$L$-regular} if the simplices are $r$-bilipschitz to a standard
simplex with edge length $1/L$, for some $r$ which perhaps depends on dimension.
The
most
common notion of subdivision used is barycentric subdivision, which is not
regular---the simplices get progressively skinnier.  However, several regular
subdivision schemes are available, including the following:
\begin{itemize}[leftmargin=*]
\item Add a central vertex to each $k$-simplex to subdivide it into $k+1$ cubes.
  Cubulate each such cube at scale $1/L$, then break each small cube into
  simplices in a standard way.  This method was described in \cite{FWPNAS}.
\item Slice each simplex into approximately $L$ slices of equal width along
  planes parallel to each face.  This subdivides it into a finite number of types
  of polyhedra.  Apply a standard simplicial subdivision to each.  A specific
  such method is given in \cite{EdGr}; its advantage is that $r$ can be taken to
  be a constant $\sqrt{2}$, not depending even on dimension.
\end{itemize}
Such subdivisions are useful for simplicial approximation of maps.
\begin{prop}[Quantitative simplicial approximation theorem]
  \label{prop:qSAT}
  For finite simplicial complexes $X$ and $Y$ with piecewise linear metrics,
  there is a constant $C$ such that any $L$-Lipschitz map $f:X \to Y$ has a
  $C(L+1)$-Lipschitz simplicial approximation via a homotopy of thickness
  $C(L+1)$ and length $C$.
\end{prop}
The main purpose of simplicial approximation in this paper, as in \cite{CDMW} and
\cite{CMW}, is to bound the behaviors of maps on simplices.  If there are only
finitely many things a map can do on a simplex, we can bound its Lipschitz
constant by the maximum Lipschitz constant of these restrictions.  Therefore it
is useful to extract this more general property and give it a name.

Let $\mathcal{F}_k$ be a finite set of maps $\Delta^k \to Y$, for some space $Y$.
If $X$ is a simplicial complex, a map $f:X \to Y$ is \emph{$\mathcal{F}$-mosaic}
if all of its restrictions to $k$-simplices are in $\mathcal{F}_k$.  Here,
$\mathcal{F}$ is a semi-simplicial set whose simplices in degree $k$ are
$\mathcal{F}_k$, which can be formed naturally via restriction maps.  We can
think of this \emph{shard complex} as a finite subcomplex of the singular
simplicial set of $Y$.

We refer to a collection of maps as \emph{uniformly mosaic} if they are all
$\mathcal{F}$-mosaic with respect to a fixed unspecified shard complex
$\mathcal{F}$.

The main advantage of this definition is that the property of being mosaic is
preserved under postcomposition.  Thus for example if we have a homotopy
equivalence $\ph:Z \to Y$ from a simplicial complex $Z$ to a cell complex $Y$
which contracts the 1-skeleton, then we can simplicially approximate a map
$X \to Z$, then compose with $\ph$ to get an $\mathcal{F}_\ph$-mosaic map for
some fixed $\mathcal{F}_\ph$ whose 1-skeleton is a point.

\subsection{Quantitative antidifferentiation}

De Rham algebras exist in several variations, including smooth and piecewise
polynomial.  In this paper, we also use the algebra of simplexwise smooth forms
on a simplicial complex.  This has several advantages: such forms can be built
skeleton-by-skeleton and this is the natural context for pullbacks of smooth
forms on a manifold by simplexwise smooth functions.  Given a simplicial complex
$X$, this is the algebra we will denote $\Omega^*X$; for a manifold with
boundary the same notation will denote the smooth forms.

A key step in both Gromov's earliest arguments in \cite{GrHED} and in this paper
is quantitative antidifferentiation of forms: given an exact $k$-form with
$L^\infty$-norm $B$, produce a $(k-1)$-form which it bounds with $L^\infty$-norm
$CB$, with $C$ depending on the space and perhaps some other requirements we
impose.  Gromov sketches an algorithm for this using quantitative Poincar\'e
lemmas to build antidifferentials skeleton by skeleton, and this was explained in
greater detail in Joshua Maher's unpublished thesis \cite{Maher}.  We give a full
proof of a similar approach here.

A duality theorem shows that this isoperimetric inequality is closely related to
the Federer--Fleming isoperimetric inequality for currents in $X$.  This kind of
duality was previously explored in \cite{CDMW}.

\subsubsection*{Quantitative Poincar\'e lemmas}

The goal of this subsection is to prove the following.
\begin{lem} \label{lem:IP}
  Let $A \subset X$ be a finite simplicial pair with the standard simplexwise
  metric.  We use the notation $\Omega^*(X,A)$ to denote forms whose restriction
  to $A$ is zero.  Then for every $k$ there is a constant $C(k,X,A)$ such that
  for every exact form $\omega \in d\Omega^{k-1}(X,A)$, there is a form
  $\alpha \in \Omega^{k-1}(X,A)$ with $d\alpha=\omega$ and
  $\lVert\alpha\rVert_\infty \leq C(k,X,A)\lVert\omega\rVert_\infty$.
\end{lem}
In order to prove this, we first show two important special cases which will also
be used later in the paper.
\begin{lem}[First quantitative Poincar\'e lemma]
  For every $0<k \leq n$, there is a constant $C_{n,k}$ such that the following
  holds.  Let $\omega \in \Omega^k(\Delta^n,\partial\Delta^n)$ be a closed smooth
  $k$-form which restricts to zero on the boundary of the standard simplex.  (If
  $k=n$, we require in addition that $\int_{\Delta^n} \omega=0$.)  Then there is a
  form $\alpha \in \Omega^{k-1}(\Delta^n,\partial\Delta^n)$ such that
  $d\alpha=\omega$ and
  $\lVert\alpha\rVert_\infty \leq C_{n,k}\lVert\omega\rVert_\infty$.
\end{lem}
\begin{proof}
  We prove this by induction on $n$ and $k$, keeping $n-k$ constant.  We note
  also that instead of the simplex we can use the unit $n$-cube, which is
  diffeomorphic to it.

  The lemma is clear for $k=0$, since then $\omega$ is the zero function.  To do
  the inductive step, we use the usual proof of the Poincar\'e lemma with compact
  support, following \cite[\S1.4]{BottTu}.  Fix a smooth bump function
  $\epsi:[0,1] \to [0,1]$ which is $0$ near $0$ and $1$ near $1$.  By applying
  the lemma one dimension lower, we get a $(k-2)$-form $\eta$ on the $(n-1)$-cube
  with $\lVert\eta\rVert_\infty \leq C_{n-1,k-1}\lVert\omega\rVert_\infty$ and
  $d\eta=\int_0^1\omega$, the fiberwise integral of $\omega$ along the first
  coordinate $x_1$.  Then
  $$\omega=d\left({\textstyle\int_0^t\omega}-\epsi(x_1)
  \pi^*({\textstyle \int_0^1\omega})-d\epsi(x_1) \wedge \pi^*\eta\right),$$
  where $\pi$ is the projection to the $(n-1)$-cube along $x_1$.  This form
  restricts to zero on the boundary of the $n$-cube and its $\infty$-norm is
  bounded by $(2+C_{n-1,k-1}\lVert d\epsi \rVert_\infty)\lVert\omega\rVert_\infty.$
\end{proof}
From here, we show how to extend nonzero forms.
\begin{lem}[Second quantitative Poincar\'e lemma]
  For every $0<k \leq n$, there is a constant $C_{n,k}$ such that the following
  holds.  Let $\omega \in \Omega^k(\Delta^n)$ be a closed $k$-form, and let
  $\alpha_\partial \in \Omega^{k-1}(\partial\Delta^n)$ be a $(k-1)$-form such that
  $d\alpha_\partial=\omega|_{\partial\Delta^n}$.  (If $n=k$, we also require that the
  pair satisfies Stokes' theorem, that is,
  $\int_{\Delta^k} \omega=\int_{\partial\Delta^k} \alpha_\partial$.)  Then there is a
  $(k-1)$-form $\alpha \in \Omega^{k-1}(\Delta^n)$ extending $\alpha_\partial$ such
  that $d\alpha=\omega$ and $\lVert\alpha\rVert_\infty \leq
  C_{n,k}(\lVert\omega\rVert_\infty+\lVert\alpha_\partial\rVert_\infty)$.
\end{lem}
\begin{proof}
  Let $U$ be the $1/(2n)$-neighborhood of $\partial\Delta^n$ in $\Delta^n$, and
  let $\ph:U \to \partial\Delta^n$ be a smooth projection with Lipschitz constant
  $L_\pi$.  Let $\epsilon:\Delta^n \to [0,1]$ be a smooth bump function with
  Lipschitz constant $L_\epsilon$ which is $1$ on $\partial\Delta^n$ and $0$
  outside $U$.  Then $\epsilon\pi^*\alpha_\partial$ is an extension of
  $\alpha_\partial$ to $\Delta^n$ with
  \begin{align*}
    \lVert \epsilon\pi^*\alpha_\partial \rVert_\infty
    &\leq L_\pi^{k-1}\lVert\alpha_\partial\rVert_\infty \\
    \lVert d(\epsilon\pi^*\alpha_\partial) \rVert_\infty &= \lVert d\epsilon \wedge
    \pi^*\alpha_\partial+\epsilon \pi^*d\alpha_\partial\rVert_\infty \leq
    L_\epsilon L_\pi^{k-1}\lVert\alpha_\partial\rVert_\infty
    +L_\pi^k\lVert\omega\rVert_\infty.
  \end{align*}
  Now we apply the previous lemma to $\omega-d(\epsilon\pi^*\alpha_\partial)$ to
  get an $\alpha^\prime \in \Omega^k(\Delta^n,\partial\Delta^n)$ with
  $$\lVert\alpha^\prime\rVert_\infty \leq C_{n,k}\bigl(L_\epsilon L_\pi^{k-1}
  \lVert\alpha_\partial\rVert_\infty+(L_\pi^k+1)\lVert\omega\rVert_\infty\bigr).$$
  The form we are looking for is
  $\alpha=\alpha^\prime+\epsilon\pi^*\alpha_\partial$.
\end{proof}
Finally, we are ready to prove Lemma \ref{lem:IP}.
\begin{proof}
  First, let $w \in C^k(X,A)$ be the simplicial $k$-cochain given by integrating
  $\omega$ over simplices.  By the De Rham theorem, this is a coboundary, and
  since the space of such coboundaries is finite-dimensional, there is an
  isoperimetric constant $c_0(k,X,A)$ and an $a \in C^{k-1}(X,A)$ with
  $\delta a=w$ and
  $$\lVert a \rVert_\infty \leq c_0(k,X,A)\lVert w \rVert_\infty \leq
  c_0(k,X,A)\vol(\Delta^k)\lVert\omega\rVert_\infty.$$
  Now we build a corresponding form $\alpha \in \Omega^k(X,A)$ by skeleta.  On
  the $(k-1)$-skeleton, we take $\alpha=a\ph d\vol$ where $\ph$ is a bump
  function with integral $1$.  We then extend inductively to each higher skeleton
  by the previous lemma.  At each step, the isoperimetric constant is multiplied
  by a constant depending only on the dimension.
\end{proof}

\subsubsection*{Isoperimetric duality}

In this section we show that the optimal isoperimetric constant of Lemma
\ref{lem:IP} is equal to another, better-known isoperimetric constant. In
geometric measure theory, a \emph{$k$-dimensional current} is simply a functional
on the space of smooth differential $k$-forms, with a boundary operator
$\partial$ defined to be dual to the differential.  The \emph{mass} of a current
$T$, which may of course be infinite, is defined by
$\mass(T)=\sup_{\lVert\omega\rVert_\infty=1} T(\omega)$.  Thus the space of currents of
finite mass is dual to $(\Omega^n(X),\lVert\cdot\rVert_\infty)$.  A
\emph{normal current} is a current $T$ such that $T$ and $\partial T$ both have
finite mass; in particular, any current of finite mass which is a cycle is
normal.  The space of normal $k$-currents in $X$ is denoted $\mathbf{N}_k(X)$.
For a simplicial pair $A \subset X$, we also define
$\mathbf{N}_k(X,A)=\mathbf{N}_k(X)/\mathbf{N}_k(A)$, equipped with the quotient
norm.  Then the following is a dual statement to Lemma \ref{lem:IP}:
\begin{lem} \label{lem:curIP}
  Let $A \subset X$ be a finite simplicial pair.  Then there is a constant
  $C(k,X,A)$ such that every normal current $T \in \mathbf{N}_{k-1}(X,A)$ has a
  filling $S$ with $\mass S \leq C\mass T$.
\end{lem}
This is a version of the Federer--Fleming isoperimetric inequality,
\cite[Thm.~5.5]{FF}.  In their original theorem, Federer and Fleming show that a
$k$-current of mass $T$ in $\mathbb{R}^n$ whose boundary is in the $k$-skeleton
of the unit cubical lattice can be pushed to a linear combination of $k$-cubes of
this lattice through a $(k+1)$-current of mass at most $C_{n,k}\mass T$;
moreover, the resulting cubical $k$-chain has mass at most $C_{n,k}\mass T$ as
well.  Except for the precise constants, their proof can be used to push a
current in a simplicial complex to its $k$-skeleton.  Since it works by
inductively pushing the current onto lower skeleta, it also works for a relative
current (when you reach $A$, stop pushing.)  Finally, once we have deformed our
current to a simplicial boundary in $(X,A)$, it is nullhomologous in a bounded
way simply because the space of simplicial boundaries $B_k(X,A)$ is
finite-dimensional.

The fact that the constants in Lemmas \ref{lem:IP} and \ref{lem:curIP} are equal
is a consequence of the Hahn--Banach theorem.  We can state this in a more
general form:
\begin{thm}[Isoperimetric duality] \label{IPD}
  Let $(V,\lVert\cdot\rVert_V)$ and $(W,\lVert\cdot\rVert_W)$ be normed vector
  spaces and $\ph:V \to W$ a (not necessarily continuous) linear operator.  There
  is a adjoint operator $\ph^*:\Omega \to V^*$ where $\Omega \subseteq W^*$ is
  the space of operators $\omega:W \to \mathbb{R}$ such that $\omega$ and
  $\ph^*\omega$ are both bounded.  Let $C_1$ and $C_2$ be the least constants
  such that:
  \begin{enumerate}
  \item For every $\epsi>0$, every $w \in \img(\ph)$ has a preimage $v$ with
    $\lVert v \rVert_V \leq C_1\lVert w \rVert_W+\epsi$.
  \item For every $\epsi>0$, every $\nu \in \img(\ph^*)$ has a preimage $\omega$
    with $\lVert \omega \rVert_{W^*} \leq C_2\lVert \nu \rVert_{V^*}+\epsi$.
  \end{enumerate}
  If $C_1$ and $C_2$ are both finite, then $C_1=C_2$.
\end{thm}
I would like to thank the referee for pointing out the need to assume the
finiteness of $C_1$ and $C_2$.
\begin{proof}
  Consider the bounded operators
  $$\ph^{-1}:(\ph(V),\lVert\cdot\rVert_W) \to (V/\ker \ph,
  \lVert\cdot\rVert_{\inf}),$$
  where $\displaystyle\lVert \bar v \rVert_{\inf}=
  \inf_{v \in \bar v} \lVert v \rVert_V$, and
  $$(\ph^*)^{-1}:(\ph^*\Omega, \lVert\cdot\rVert_{V^*}) \to
  (W^*/\ker\ph^*, \lVert\cdot\rVert_{\inf}),$$
  where $\lVert\bar\omega\rVert_{\inf}=
  \displaystyle\inf_{\omega \in \bar\omega} \lVert\omega\rVert_{W^*}$.

  Here, $\ph^{-1}$ is a bounded isomorphism of vector spaces, but not necessarily
  a bilipschitz equivalence; $(\ph^*)^{-1}$ is injective but its image
  $\Omega/\ker\ph^* \subseteq W^*/\ker\ph^*$ is not necessarily the whole space.
  Then $C_1$ and $C_2$ are the operator norms of $\ph^{-1}$ and $(\ph^*)^{-1}$.  It
  is therefore enough to prove that $\ph^{-1}$ and $(\ph^*)^{-1}$ are adjoint
  operators on dual normed vector spaces and so have the same norm.

  First, any $\bar\omega \in W^*/\ker\ph^*$ gives a well-defined operator on
  $\ph(V)$: if $\bar\omega=\overline{\omega^\prime}$, then
  $\ph^*(\omega-\omega^\prime)=0$ and so $\langle\omega,\ph(v)\rangle=
  \langle\omega^\prime,\ph(v)\rangle$.  Conversely, by the Hahn--Banach theorem,
  any functional $\omega_0:\ph(V) \to \mathbb{R}$ which is continuous with
  respect to $\lVert\cdot\rVert_W$ has an extension to $W^*$.  Thus
  $(W^*/\ker\ph^*,\lVert\cdot\rVert_{\inf})$ is the dual normed space to
  $(\ph(V),\lVert\cdot\rVert_W)$.  A similar argument holds for the other pair,
  though one needs to invoke (1) to show the duality.  Finally, it is clear that
  $$\langle\nu,\ph^{-1}(w)\rangle=\langle(\ph^*)^{-1}\nu,w\rangle.$$
  This completes the proof.
\end{proof}

\section{Homotopy theory of DGAs} \label{S:Obs}

In this section we sketch out the homotopy theory of differential graded
algebras, following the treatment of \cite[Ch.~IX and X]{GrMo}.  The relatively
explicit formulation helps us obtain quantitative bounds on the sizes of DGA
homotopies, which we will later harness to obtain various geometric bounds.  We
also review Gromov's arguments bounding the homotopy classes of maps with a given
Lipschitz constant.

A \emph{(commutative) differential graded algebra} (DGA) will always denote a
cochain complex of $\mathbb{Q}$- or $\mathbb{R}$-vector spaces equipped with a
graded commutative multiplication which satisfies the (graded) Leibniz rule.  The
prototypical example of an $\mathbb{R}$-DGA is the algebra of smooth forms on a
manifold or piecewise smooth forms on a simplicial complex.

The cohomology of a DGA is the cohomology of the underlying cochain complex.  The
relative cohomology of a DGA homomorphism $\ph:\mathcal{A} \to \mathcal{B}$ is
the cohomology of the cochain complex
$$C^n(\ph)=\mathcal{A}^n \oplus \mathcal{B}^{n-1}$$
with the differential given by $d(a,b)=(da,\ph(a)-db)$.  This cohomology fits, as
expected, into an exact sequence involving $H^*(\mathcal{A})$ and
$H^*(\mathcal{B})$.

Given a coefficient vector space $V$, $H^*(\mathcal{A},V)$ is the cohomology of
the cochain complex $\Hom(V,\mathcal{A}^n)$.  By the universal coefficient
theorem, this is naturally isomorphic to $\Hom(V,H^*(\mathcal{A}))$, but we will
frequently be using the cochain complex itself.

A \emph{weak equivalence} between DGAs $\mathcal{A}$ and $\mathcal{B}$ is a
homomorphism $\mathcal{A} \to \mathcal{B}$ which induces an isomorphism on
cohomology.

An algebra $\mathcal{A}$ is \emph{simply connected} if $\tilde H^0(\mathcal{A})=
H^1(\mathcal{A})=0$.  If $\mathcal{A}$ is simply connected and of
\emph{finite type} (i.e.\ it has finite-dimensional cohomology in every degree)
then it has a \emph{minimal model}: a weak equivalence
$m_{\mathcal{A}}:\mathcal{M_A} \to \mathcal{A}$ where $\mathcal{M_A}$ is freely
generated as an algebra by finite-dimensional vector spaces $V_n$ in degree $n$
(we write
$$\mathcal{M_A}=\bigwedge_{n=2}^\infty V_n)$$
and the differential satisfies
$$dV_n \subseteq \bigwedge_{k=2}^{n-1} V_k.$$
In other words, $\mathcal{M_A}$ can be built up via a sequence of
\emph{elementary extensions} (sometimes called \emph{Hirsch extensions})
$$\mathcal{M_A}(n+1)=\mathcal{M_A}(n)\langle V_{n+1} \rangle,$$
with the differential on $\mathcal{M_A}(n+1)$ extending that on
$\mathcal{M_A}(n)$, starting with $\mathcal{M_A}(1)=\mathbb{Q}$ or $\mathbb{R}$.
We refer to elements of the $V_n$ as \emph{indecomposables}.  We will often
describe finitely generated free DGAs by indicating the degree of generators as
superscripts in parentheses: $a^{(3)}$ means that $a$ is a generator in degree 3.

In particular, if $Y$ is a simply connected manifold or simplicial complex, the
algebra of forms $\Omega^*Y$ has a minimal model which we will call
$m_Y:\mathcal{M}_Y^* \to \Omega^*Y$.  This models the Postnikov tower of $Y$:
each $V_n \cong \Hom(\pi_n(Y),\mathbb{R})$ and the differential on $V_n$ is dual
to the $k$-invariant of the fibration $Y_{(n)} \to Y_{(n-1)}$.  This can be shown
inductively via obstruction theory.

\subsection{Obstruction theory}

Given a principal fibration $K(\pi,n) \to E \xrightarrow{p} B$ and a space $X$,
obstruction theory gives an exact sequence of sets
$$H^n(X;\pi) \to [X,E] \xrightarrow{p_*} [X,B] \xrightarrow{\mathcal{O}}
H^{n+1}(X;\pi),$$
in the sense that $\img p_*=\mathcal{O}^{-1}(0)$ and $H^n(X;\pi)$ acts on $[X,E]$
via an action whose orbits are exactly the preimages of classes in $[X,B]$.
Moreover, if $B$ is simply connected (or more generally, $\pi_1(B)$ acts
homotopically trivially on the fiber) then over a given map $f:X \to B$, there is
an exact sequence of groups
$$\cdots \to H^{n-1}(X;\pi) \to \pi_1(E^X,\tilde f) \to \pi_1(B^X,f)
\to H^n(X;\pi) \to p_*^{-1}([f]) \to 0,$$
where $p_*^{-1}([f])$, the set of homotopy classes of maps lifting $f$, is a
torsor acted on by $H^n(X;\pi)$ and $\tilde f$ is any lift of $f$.

We now give DGA versions of these statements.  First define homotopy of DGA
homomorphisms as follows: $f,g:\mathcal{A} \to \mathcal{B}$ are homotopic if
there is a homomorphism
$$H:\mathcal{A} \to \mathcal{B}\otimes\mathbb{R}\langle t^{(0)},dt^{(1)} \rangle$$
such that $H|_{\substack{t=0\\dt=0}}=f$ and $H|_{\substack{t=1\\dt=0}}=g$.  We think of
$\mathbb{R}\langle t,dt \rangle$ as an algebraic model for the unit interval and
this notion as an abstraction of the map induced by an ordinary smooth homotopy.
In particular, it defines an equivalence relation \cite[Cor.~10.7]{GrMo}.
Moreover, for any piecewise smooth space $X$ there is a map
$$\rho:\Omega^*X \otimes \mathbb{R}\langle t,dt \rangle \to
\Omega^*(X \times [0,1])$$
given by ``realizing'' this interval, that is, interpreting the $t$ and $dt$ the
way one would as forms on the interval.  We will use this realization map further
in the paper.

We also introduce some notation which is useful for constructing homotopies
between DGA homomorphisms.  For any DGA $\mathcal A$, define an operator
$\int_0^t:\mathcal A \otimes \mathbb{R}\langle t,dt \rangle \to
\mathcal A \otimes \mathbb{R}\langle t,dt \rangle$ by
$${\textstyle\int_0^t a \otimes t^i}=0, {\textstyle\int_0^t a \otimes t^idt}=
(-1)^{\deg a}a \otimes \frac{t^{i+1}}{i+1}$$
and an operator $\int_0^1:\mathcal A \otimes \mathbb{R}\langle t,dt \rangle \to
\mathcal A$ by
$${\textstyle\int_0^1 a \otimes t^i}=0,{\textstyle\int_0^1 a \otimes t^idt}=
(-1)^{\deg a}\frac{a}{i+1}.$$
These provide a formal analogue of fiberwise integration; in particular, they
satisfy the identities
\begin{align}
  d\bigl({\textstyle\int_0^t} u\bigr)+{\textstyle\int_0^t} du &=
  u-u|_{\substack{t=0\\dt=0}} \otimes 1 \label{eqn:I0t} \\
  d\bigl({\textstyle\int_0^1} u\bigr)+{\textstyle\int_0^1} du &=
  u|_{\substack{t=1\\dt=0}}-u|_{\substack{t=0\\dt=0}}. \label{eqn:I01}
\end{align}

Now we state the main lemma of obstruction theory, which states the conditions
under which a map can be extended over an elementary extension.
\begin{prop}[10.4 in \cite{GrMo}] \label{extExist}
  Let $\mathcal{A}\langle V \rangle$ be an $n$-dimensional elementary extension
  of a DGA $\mathcal{A}$.  Suppose we have a diagram of DGAs
  $$\xymatrix{
    \mathcal{A} \ar[r]^f \ar@{^{(}->}[d] & \mathcal{B} \ar[d]^h \\
    \mathcal{A}\langle V \rangle \ar[r]^g & \mathcal{C}
  }$$
  with $g|_{\mathcal{A}} \simeq hf$ by a homotopy $H:\mathcal{A} \to \mathcal{C}
  \otimes \mathbb{R}\langle t,dt \rangle$.  Then the map
  $O:V \to \mathcal{B}^{n+1} \oplus \mathcal{C}^n$ given by
  $$O(v)=\left(f(dv), g(v)+{\textstyle\int_0^1 H(dv)}\right)$$
  defines an obstruction class $[O] \in H^{n+1}(h:\mathcal{B} \to \mathcal{C};V)$
  to producing an extension $\tilde f:\mathcal{A}\langle V \rangle \to
  \mathcal{B}$ of $f$ with $g \simeq h \circ \tilde f$ via a homotopy $\tilde H$
  extending $H$.
\end{prop}
When the obstruction vanishes, there are maps
$(b,c): V \to \mathcal{B}^n \oplus \mathcal{C}^{n-1}$ such that $d(b,c)=O$, i.e.
\begin{align*}
  db(v) &= f(dv) \\
  dc(v) &= h\circ b(v)-g(v)-{\textstyle\int_0^1 H(dv)}.
\end{align*}
Then for $v \in V$ we can set $\tilde f(v)=b(v)$ and
\begin{equation} \label{extHtpy}
  \tilde H(v)=g(v)+d(c(v) \otimes t)+{\textstyle\int_0^t H(dv)}.
\end{equation}
This gives a specific formula for the extension.

This lemma has a relative analogue which is also quite useful.
\begin{prop}[10.5 in \cite{GrMo}] \label{extRel}
  Let $\mathcal{A}\langle V \rangle$ be an $n$-dimensional elementary extension
  of a DGA $\mathcal{A}$.  Suppose we have a diagram of DGAs
  $$\xymatrix{
    \mathcal{A} \ar[r]^f \ar@{^{(}->}[d] & \mathcal{B} \ar[d]^h \ar[r]^\mu &
    \mathcal{D} \\
    \mathcal{A}\langle V \rangle \ar[r]^g & \mathcal{C} \ar[ru]^\nu
  }$$
  where
  \begin{enumerate}
  \item $g|_{\mathcal{A}} \simeq hf$ by a homotopy $H:\mathcal{A} \to \mathcal{C}
    \otimes \mathbb{R}\langle t,dt \rangle$ such that $\nu \circ H$ is constant,
  \item $\mu$ is surjective,
  \item $\nu \circ h=\mu$ on the nose, and
  \item $\mu \circ f=\nu \circ g|_{\mathcal{A}}$ on the nose.
  \end{enumerate}
  Then the map $O:V \to \mathcal{B}^{n+1} \oplus \mathcal{C}^n$ given by
  $$O(v)=\bigl(f(dv), g(v)+{\textstyle\int_0^1 H(dv)}\bigr)$$
  defines an obstruction class $[O] \in H^{n+1}(h:\mathcal{B} \to \mathcal{C};V)$
  to producing an extension $\tilde f:\mathcal{A}\langle V \rangle \to
  \mathcal{B}$ of $f$ with $g \simeq h \circ \tilde f$ via a homotopy $\tilde H$
  extending $H$, where $\nu \circ \tilde H$ is constant
  (i.e.\ $\nu \circ \tilde H=(\mu \circ \tilde f) \otimes 1$.)
\end{prop}
More specifically, in this case we can define $b$ and $c$ as above so that
$\mu \circ b=\nu \circ g|_V$ and $\nu \circ c=0$; then $\tilde H$ is again
defined via \eqref{extHtpy}.

A special case of Prop.~\ref{extRel} gives rise to the following result:
\begin{prop} \label{prop:extBoth}
  Let $\mathcal{A}\langle V \rangle$ be an $n$-dimensional elementary extension
  of a DGA $\mathcal{A}$.  Suppose we have maps
  $$\mathcal{A}\langle V \rangle \xrightarrow{\ph,\psi} \mathcal{M}
  \xrightarrow{\mu} \mathcal{N}$$
  with $\mu$ surjective, together with a homotopy
  $\Phi:\mathcal{A} \to \mathcal{M} \otimes \langle t,dt \rangle$ between
  $\ph|_{\mathcal{A}}$ and $\psi|_{\mathcal{A}}$ and a homotopy
  $\chi:\mathcal{A}\langle V \rangle \to \mathcal{N}\langle t,dt \rangle$ between
  $\mu \circ \ph$ and $\mu \circ \psi$ which extends $\mu \circ \Phi$.  Then the
  obstruction in $H^n(\mu:\mathcal{M} \to \mathcal{N};V)$ to producing a homotopy
  $$\tilde \Phi:\mathcal{A}\langle V \rangle \to \mathcal{M} \otimes
  \langle t,dt \rangle$$
  which extends $\Phi$ and lifts $\chi$ is given by
  $O(v)=\bigl(\psi(v)-\ph(v)-\int_0^1 \Phi(dv),\int_0^1 \chi(v)\bigr)$.
\end{prop}
\begin{proof}
  We apply Prop.~\ref{extRel} using
  \begin{align*}
    \mathcal{B} &= \mathcal{M} \otimes \langle t,dt \rangle; \quad f=\Phi \\
    \mathcal{C}=\mathcal{D} &= \mathcal{M} \otimes \langle t,dt \rangle/
                  \ker \mu \otimes \langle t(1-t),dt\rangle \\
    g &= \ph \otimes (1-t)+\psi \otimes t+\chi
    -(\chi|_{t=0} \otimes (1-t)+\chi_{t=1} \otimes t).
  \end{align*}
  Thus we obtain an obstruction cocycle
  $\hat O:V \to \mathcal{B}^{n+1} \oplus \mathcal{C}^n$ given by
  $$\hat O(v)=(\Phi(dv),g(v)).$$
  We can get a cocycle $(p,q)$ which is cohomologous to $\hat O$ and satisfies
  $p|_{t=0}=p|_{t=1}=0$ and
  $$q=\bigl(\chi-\chi|_{t=0} \otimes (1-t)-\chi|_{t=1} \otimes t\bigr)|_V$$
  by subtracting off
  $$d(\ph|_V \otimes (1-t)+\psi|_V \otimes t,0).$$
  Finally, by \eqref{eqn:I0t} and \eqref{eqn:I01}, $(p,q)$ is cohomologous to
  $(-1)^n\bigl(({\textstyle \int_0^1 p}) \otimes dt,
  -({\textstyle \int_0^1 q}) \otimes dt\bigr)$ via
  $$d\bigl(-{\textstyle \int_0^t p}+({\textstyle \int_0^1 p}) \otimes t,
  {\textstyle \int_0^t q}-({\textstyle \int_0^1 q}) \otimes t\bigr).$$
  It is easy to see that $\int_0^1 q(v)=\int_0^1 \chi(v)$ and
  $${\textstyle \int_0^1 p(v)}=-\psi(v)+\ph(v)+{\textstyle \int_0^1 \Phi(dv)}.$$
  This obstruction is zero if and only if $O \in H^n(\mu;V)$ is zero. 
\end{proof}
Prop.~\ref{prop:extHtpy} will give a quantitative version.

Finally, we give the DGA version of the exact sequence of groups; in this, unlike
the previous lemmas, the domain algebra must be minimal.
\begin{prop}
  Let $\mathcal{A}\langle V \rangle$ be an $n$-dimensional elementary extension
  of a minimal DGA $\mathcal{A}$.  Then for any DGA $\mathcal{B}$ and map
  $\ph:\mathcal{A} \to \mathcal{B}$ which has an extension
  $\tilde \ph:\mathcal{A}\langle V \rangle \to \mathcal{B}$, there is an exact
  sequence of groups (and a torsor)
  $$[\mathcal{A}\langle V \rangle,\mathcal{B} \otimes
    \mathbb{R}\langle e^{(1)} \rangle]_{\tilde\ph} \to [\mathcal{A},
    \mathcal{B} \otimes \mathbb{R}\langle e^{(1)} \rangle]_\ph
  \xrightarrow{\mathcal{O}} H^n(\mathcal{B};V) \to
  \left\{\begin{array}{c}\text{elements of }
  [\mathcal{A}\langle V \rangle,\mathcal{B}]\\
  \text{which extend }\ph\end{array}\right\}
  \to 0.$$
  Moreover, the group structure on the first two sets is given as follows.  Their
  elements are given by representatives of the form $\ph+\eta \otimes e$, where
  $\eta:\mathcal{A}^* \to \mathcal{B}^{*-1}$ (respectively
  $\mathcal{A}\langle V \rangle^* \to \mathcal{B}^{*-1}$) satisfies the identities
  $d\eta=\eta d$ and
  \begin{equation} \label{Leibniz}
    \eta(uv)=(-1)^{\deg v}\eta(u)\ph(v)+\ph(u)\eta(v).
  \end{equation}
  Then given two elements $\ph+\eta_1$ and $\ph+\eta_2$, their sum is given by
  $$(\ph+\eta_1 \otimes e)\boxplus(\ph+\eta_2 \otimes e)=\ph+(\eta_1+\eta_2)
  \otimes e.$$
  The arrow $\mathcal{O}$ is given by
  $$\ph+\eta \otimes e \mapsto \eta d|_V:V \to \mathcal{B}^n.$$
\end{prop}
Exactness at the third and fourth term are given in \cite[Prop.~14.4]{GrMo}.
Exactness at the second term can be proven using Prop.~\ref{extRel}, similarly to
Prop.~\ref{prop:extBoth}.

\subsection{Quantitative aspects}

In this subsection, $X$ will be a finite piecewise Riemannian simplicial complex
and $Y$ a compact simply-connected Riemannian manifold with boundary with minimal
model $m_Y:\mathcal{M}_Y^* \to \Omega^*Y$.  The technical results of this paper
largely concern homomorphisms $\ph:\mathcal{M}_Y^* \to \Omega^*X$ for such $X$
and $Y$.  We would like to define a notion of size on such homomorphisms.  Given
a simplexwise Riemannian metric on $X$, we equip each $\Omega^kX$ with the
$L^\infty$-norm; we also fix a norm on each of the vector spaces $V_k$ of degree
$k$ indecomposables of $\mathcal{M}_Y^*$.  Since the $V_k$ are
finite-dimensional, the choice of this norm affects anything that depends on
finitely many of them only up to a constant.  Given this data, we define the
\emph{(formal) dilatation} of $\ph$ by
$$\Dil(\ph)=
\max_{k \in \{2,\ldots,\dim X\}} \lVert \ph|_{V_k} \rVert_{\mathrm{op}}^{1/k}.$$
Note that if $f:X \to Y$ is an $L$-Lipschitz map, then $f^*$ multiplies the
$L^\infty$ norm of $k$-forms by at most $L^k$.  Therefore when $\ph=f^*m_Y$ for
some map $f:X \to Y$,
$$\Dil(\ph) \leq C\Lip f,$$
where $C$ depends only on $m_Y$ and the norms on the $V_k$.

We define the dilatation of a homotopy via the realization map $\rho$.  Since we
often want to scale the time interval independently of $X$, we define a whole
family
$$\Dil_\tau(\Phi)=\Dil(\rho_\tau\Phi)$$
where $\rho_\tau:\Omega^*X \otimes \mathbb{R}\langle t,dt \rangle \to
\Omega^*(X \times [0,\tau])$ sends $t \mapsto t/\tau$.  One can then think of
$\tau$ as the ``length'' of the formal interval.


We will frequently want to apply the obstruction lemmas in such a way that we can
say something quantitative about the extension.  We give here a couple of
specialized instances in which we can do this.

\begin{prop} \label{prop:extHtpy}
  Suppose that $\Phi_k:\mathcal{M}_Y^*(k) \to \Omega^*X \otimes
  \mathbb{R}\langle t,dt \rangle$ is a partially defined homotopy between
  $\ph,\psi:\mathcal{M}_Y^* \to \Omega^*X$.
  \begin{enumerate}[label={(\roman*)}]
  \item \label{num:obst} The obstruction to extending $\Phi_k$ to a homotopy
    $$\Phi_{k+1}:\mathcal{M}_Y^*(k+1) \to \Omega^*X \otimes
    \mathbb{R}\langle t,dt \rangle$$
    is a class in $H^{k+1}(X;V_{k+1})$ represented by a cochain $\sigma$ with (for
    any $\tau>0$)
    $$\lVert \sigma \rVert_{\mathrm{op}} \leq
    \tau C(k,d|_{V_{k+1}})\Dil_\tau(\Phi_k)^{k+2}
    +\Dil(\ph)^{k+1}+\Dil(\psi)^{k+1}.$$
  \item \label{num:ext}
    If this obstruction class vanishes, then we can choose $\Phi_{k+1}$ so that
    $$\bigl\lVert(\Phi_{k+1})_i^j|_{V_{k+1}}\bigr\lVert_{\mathrm{op}} \leq
    (C_{\mathrm{IP}}+2)\bigl(\tau C(k,d|_{V_{k+1}})
    \Dil_\tau(\Phi_k)^{k+2}+\Dil(\ph)^{k+1}+\Dil(\psi)^{k+1}\bigr),$$
    where $C_{\mathrm{IP}}$ is the isoperimetric constant for $(k+2)$-forms in $X$
    and $\tau>0$ is arbitrary.
  \end{enumerate}
  Moreover, if for some subcomplex $A \subset X$ we have an existing homotopy
  $$\chi:\mathcal{M}_Y^* \to \Omega^*A \otimes \mathbb{R}\langle t,dt \rangle$$
  between $\ph|_A$ and $\psi|_A$, then if the obstruction from
  Prop.~\ref{prop:extBoth} vanishes, we can get an extension with similar bounds,
  using a relative isoperimetric constant and with an additional
  $O(\tau\Dil_\tau(\chi))$ term.
\end{prop}
\begin{proof}
  We start with the absolute case, which is simpler.  We let
  $\sigma(v)=\psi(v)-\ph(v)-\int_0^1 \Phi_k(dv)$.  This clearly satisfies the
  given bound; by Prop.~\ref{prop:extBoth} it is the obstruction to extending
  $\Phi_k$.

  Now, suppose the obstruction class vanishes.  Then we can choose
  $c:V_{k+1} \to \Omega^*X$ such that $dc=\sigma$ and $\lVert c \rVert_\infty \leq
  C_{\mathrm{IP}}\lVert\sigma\rVert_{\mathrm{op}}$ and set
  $$\Phi_{k+1}(v)=\ph(v)+d(c(v) \otimes t)+{\textstyle \int_0^t \Phi_k(dv)},$$
  similar to \eqref{extHtpy}.  This satisfies the bound in (ii).

  We now tackle the relative case.  We must again choose $\Phi_{k+1}$ so that
  $d\Phi_{k+1}(v)=\Phi(dv)$, but this time we also need to satisfy the condition
  $\Phi_{k+1}(v)|_A=\chi(v)$.  Let $\pi:\hat A \times [0,1] \to A$ be the
  piecewise linear deformation retraction of a neighborhood of $A$ in $X$ to $A$,
  and choose a bump function $\epsi:X \to [0,1]$ which is $1$ on $A$ and
  supported on $\hat A$.  Finally, let
  $$\hat\chi(v)=\left\{\begin{array}{l l}
  \epsi\pi_1^*\chi(v) & \text{at points in }\hat A \\
  0 & \text{outside }\hat A.\end{array}\right.$$
  Then if the obstruction
  $$\bigl(\psi-\ph-{\textstyle\int_0^1\Phi_k(dv)},
  {\textstyle\int_0^1\chi(v)}\bigr)$$
  vanishes, we can choose $c:V_{k+1} \to \Omega^*X$ supported on $X \setminus A$
  and with the right isoperimetric bounds such that
  $$dc=\psi-\ph-{\textstyle \int_0^1\Phi_k(dv)}
  -d{\textstyle\int_0^1\hat\chi(v)},$$
  and then set
  $$\Phi_{k+1}(v)=\ph(v)+d(c(v) \otimes t)+{\textstyle\int_0^t \Phi_k(dv)}
  +d{\textstyle \int_0^t \hat\chi(v)}.$$
  This is the extension we are looking for.
\end{proof}
In the specific instances we consider, we can often obtain better bounds.
Suppose that $\ph$ and $\psi$ both have dilatation $\leq L$, and that we can
construct a homotopy $\Phi$ between them formally up to degree $n$ without
encountering any nonzero obstructions that make extendability dependent on
choices made in lower degree.  Write $\Phi_i^j$ for the $t^i(dt)^j$-term of
$\Phi$.  We claim that the homotopy can be built so that for $k \leq n$ and for
some constants $C(k,X,Y)$ depending on the norms on the $V_k$,
\begin{equation} \label{eqn:better}
  \lVert \Phi_i^j|_{V_k} \rVert_{\mathrm{op}} \leq C(k,X,Y)L^{2k-2}\text{, for }j=0,1.
\end{equation}
Clearly this is true for $k=2$, since degree 2 indecomposables have zero
differential.  Now suppose it's true up to $k-1$.  Then for $a \in dV_k$,
$\Phi(a)$ is a sum of some number of terms (depending on $a$) each with
$\infty$-norm bounded by
$$\prod_{r_1+\cdots+r_\ell=k+1} C(k-1,X,Y)L^{2r_i-2} \leq C(k-1,X,Y)^\ell L^{2k-2}.$$
Since $V_k$ is finite-dimensional, this gives us a $C(k,X,Y)$ which depends on
$C(k-1,X,Y)$ as well as the algebraic structure of the differentials.

Moreover, the largest power of $t$ present is bounded only as a function of $k$,
as is clear from the construction.  In particular, we end up with
$\Dil_1(\Phi) \lesssim L^{\frac{2n-2}{n}}$ and $\Dil_{L^{-2}}(\Phi) \lesssim L^2$.

A second quantitative lemma concerns formal (that is, algebraic) concatenation of
homotopies.  The proof of \cite[Cor.~10.7]{GrMo} shows in a formal way that DGA
homotopy is a transitive relation.  We reproduce this proof with quantitative
bounds on the size of the concatenation.
\begin{prop} \label{prop:cat}
  Suppose $\ph,\psi,\xi:\mathcal{M}_Y^* \to \Omega^*X$ are homomorphisms and we
  are given homotopies $\Phi$ between $\ph$ and $\psi$ and $\Psi$ between $\psi$
  and $\xi$.  Then we can find a homotopy
  $$\Xi:\mathcal{M}_Y^* \to \Omega^*X \otimes \mathbb{R}\langle t,dt \rangle$$
  between $\ph$ and $\xi$ such that for any $L$ satisfying
  $\Dil_{L^{-1}}(\Phi) \leq L$ and $\Dil_{L^{-1}}(\Psi) \leq L$,
  $$\Dil_{L^{-1}}(\Xi) \leq C(Y,\dim X)L,$$
  and moreover $\Dil(\int_0^1\Xi) \leq \Dil(\int_0^1\Phi)+\Dil(\int_0^1\Psi)$.
\end{prop}
\begin{proof}
  Roughly speaking, we will arrange $\Phi$ and $\Psi$ along two sides of a formal
  square, extend the map to the rest of the square, and then restrict to the
  diagonal to get $\Xi$.  Here is how this is done in detail.

  Write $\Phi_i^0$ and $\Psi_i^0$ for the coefficients of $t^i$ of $\Phi$ and
  $\Psi$ respectively, and $\Phi_i^1$, $\Psi_i^1$ for the coefficients of
  $t^idt$.  We first note that the formula
  $$``\Phi+\Psi\text{''}=\sum_{i=0}^{p_1} \Phi_i^0 \otimes t^i+
  \sum_{j=0}^{q_1} \Phi_j^1 \otimes t^jdt+\sum_{k=1}^{p_2} \Psi_k^0 \otimes s^k+
  \sum_{\ell=0}^{q_2} \Psi_\ell^1 \otimes s^\ell ds$$
  (where $k$ starts at 1 because the $t$ parts already restrict to $\psi$ when
  $t=1$) defines a DGA map
  $$\mathcal{M}_Y^* \to \Omega^*X \otimes \mathbb{R}\langle t,dt,s,ds \rangle/
  \langle s(t-1),(t-1)ds,sdt \rangle.$$
  This should be thought of as the DGA of two sides of a square, and we want to
  lift to a map
  $$\bar\Xi:\mathcal{M}_Y^* \to
  \Omega^*X \otimes \mathbb{R}\langle t,dt,s,ds \rangle$$
  to the DGA of the whole square.  We do this by induction on degree.  The map is
  trivial on $\mathcal{M}_Y^*(1)$; then we extend from $\mathcal{M}_Y^*(n)$ to
  $\mathcal{M}_Y^*(n+1)$ by defining
  $$\bar\Xi(v)=``\Phi+\Psi\text{''}(v)+{\textstyle\int_0^s}\bigl(\bar\Xi(dv)-
  (\bar\Xi(dv))|_{t=1}\bigr).$$
  It is easy to check that this has the right differential and the right
  restrictions to $t=1$ and to $s=0$.  Finally, we take the ``diagonal''
  $\Xi=\bar\Xi|_{s=t}$.

  Now we discuss the dilatation of $\Xi$.  The inequality for $\Dil(\int_0^1\Xi)$
  is clear since
  $${\textstyle\int_0^1\Xi}={\textstyle\int_0^1}(``\Phi+\Psi\text{''}|_{s=t})
  ={\textstyle\int_0^1\Phi}+{\textstyle\int_0^1\Psi}.$$
  For the other inequality, we can assume without loss of generality that $L=1$;
  we can achieve the conditions by scaling the metric on $X$ by $L$.  Thus we
  need to show that if $\Dil_1(\Phi) \leq 1$ and $\Dil_1(\Psi) \leq 1$, then
  $$\Dil_1(\Xi) \leq C(Y,\dim X).$$
  Clearly it is enough to bound the dilatation of the realization of $\bar\Xi$ as
  a map $\mathcal{M}_Y^* \to \Omega^*X \times [0,1]^2$.

  This can again be done by induction on degree.  We need only remark that:
  \begin{align*}
    \lVert\alpha\beta\rVert_\infty
    &\leq \lVert\alpha\rVert_\infty \cdot \lVert\beta\lVert_\infty \\
    \lVert \textstyle{\int_0^s}\omega \rVert_\infty
    &\leq \lVert\omega\rVert_\infty.
  \end{align*}
  This allows us to bound $\lVert \bar\Xi|_{V_k} \rVert_{\mathrm{op}}$ in terms of
  the operator norms in lower degrees and the structure of $\mathcal{M}_Y^*$.
  Thus the final constant we get depends on $Y$ and the dimension of $X$.
\end{proof}

\subsection{Homotopy periods and Gromov's results}\label{S:pi_k}

Now let $f:S^n \to Y$ be a smooth map.  If we attempt to nullhomotope
$f^*m_Y:\mathcal{M}^*_Y \to \Omega^*S^n$ by the method of
Prop.~\ref{prop:extHtpy}, the procedure does not fail until the very last step,
where the obstruction
$$\alpha \in H^{n+1}(\Omega^*S^n \otimes \mathbb{R}\langle t,dt \rangle \to
\Omega^*S^n \otimes \mathbb{R}\langle t \rangle/\langle t(1-t) \rangle;V_n) \cong
H^n(S^n;V_n)$$
given by the formula $\alpha=\bigl[f^*m_Y|_{V_n}+\int_0^1 \Phi_nd|_{V_n}\bigr]$ may
be nontrivial.  This obstruction determines an element of $\pi_n(Y) \otimes
\mathbb{R} \cong \Hom(V_n,\mathbb{R})$ and can be computed algorithmically by
repeated antidifferentiation.  In other words, it generalizes Whitehead's
construction of the Hopf invariant and coincides with the construction of
``homotopy periods'' outlined by Sullivan in \cite[\S11]{SulLong}.

\begin{ex} \label{exs:pi_n}
  \begin{enumerate}[leftmargin=*,label={(\roman*)}]
  \item Let $f:S^3 \to S^2$ be a map.  The minimal DGA of $S^2$ is given by
    $$\langle x^{(2)},y^{(3)} \mid dx=0,dy=x^2 \rangle;$$
    clearly, $m_{S^2}y=0$ in any minimal model.  Thus the first stage of a
    nullhomotopy of $f^*m_{S^2}$ is given by
    $$\Phi_2(x)=f^*m_{S^2}x \otimes (1-t)+c(x) \otimes dt$$
    where $dc(x)=f^*m_{S^2}x$.  The obstruction to extending this to $y$ is given
    by $-[f^*m_{S^2}x \wedge c(x)] \in H^3(S^3;\mathbb{R})$.  Up to sign, this is
    the Hopf invariant.
  \item For a slightly more complicated example, we consider $Y=(S^3 \times S^3)
    \setminus D^6$.  This is a 6-manifold homotopy equivalent to $S^3 \vee S^3$;
    the relevant part of the minimal model is
    $$\langle x_1^{(3)},x_2^{(3)},y^{(5)},z_1^{(7)},z_2^{(7)},\ldots \mid
    dx_i=0,dy=x_1x_2,dz_i=x_iy,\ldots \rangle.$$
    Consider a map $f:S^7 \to Y$; again we try to nullhomotope $f^*m_Y$, and the
    first few stages are
    \begin{align*}
      \Phi(x_i) &= f^*m_Yx_i \otimes (1-t)-c(x_i) \otimes dt \\
      \Phi(y) &= f^*m_Yy \otimes (1-t)+\frac{1}{2}(f^*m_Yx_1 \wedge c(x_2)
      -c(x_1) \wedge f^*m_Yx_2) \otimes (t-t^2)-c(y) \otimes dt,
    \end{align*}
    where $dc(x_i)=f^*m_Yx_i$ and
    $$dc(y)=f^*m_Yy+\frac{1}{2}(f^*m_Yx_1 \wedge c(x_2)
    -c(x_1) \wedge f^*m_Yx_2).$$
    Since $m_Yz_i=0$ for dimension reasons, the obstruction to extending to $z_i$
    is given by
    $$-c(x_i) \wedge \Bigl(\frac{1}{2}f^*m_Yy+\frac{1}{12}(f^*m_Yx_1 \wedge c(x_2)
    -c(x_1) \wedge f^*m_Yx_2)\Bigr)+\frac{1}{2}f^*m_Yx_i \wedge c(y).$$
    Clearly, homotopy periods quickly become impractical to compute by hand for
    more complicated DGAs.  Similar examples were computed by Richard Hain in his
    PhD thesis \cite{HainBk}.
  \end{enumerate}
\end{ex}

Suppose now that $f$ is $L$-Lipschitz.  Then by \eqref{eqn:better}, we can make
sure that the obstruction class $\alpha$ satisfies
$\lVert\alpha\rVert_{\mathrm{op}} \lesssim L^{2n-2}$.  Since the map
$\pi_n(X) \to H^n(S^n;V_n)$ is a group homomorphism with finite kernel and
covolume, this proves the following results of Gromov:
\begin{thm} Let $Y$ be simply connected and Lipschitz homotopy equivalent to a
  finite complex.
  \begin{enumerate}[leftmargin=*,label={(\roman*)}]
  \item The distortion function of an element of $\pi_n(Y)$ is
    $\Omega(k^{1/(2n-2)})$.
  \item The growth function of $\pi_n(Y)$ is polynomial and in fact
    $O(L^{(2n-2)\rk(\pi_n(Y) \otimes \mathbb{Q})})$.
  \end{enumerate}
\end{thm}
While these were stated in various combinations in \cite{GrHED}, \cite{GrQHT},
and \cite[Ch.~7]{GrMS}, the proofs are essentially omitted in the first two and
incorrect in the last.  This section is meant to close this gap.

The bounds above, however, are not sharp in most cases.  We do not currently know
how to express in full generality the bounds whose sharpness is presumed by
Gromov's conjectures.  They are obtained by assuming that all pullbacks of
genuine $k$-forms on $Y$ have $L^\infty$ norm $\lesssim L^k$ and inducting to
obtain bounds on other forms; on the other hand, the algorithm given at the
beginning of this subsection (which coincides with that given by Sullivan and is
at least weakly canonical) does not always produce the optimal exponent.  We
illustrate this by way of yet another example.
\begin{ex}
  Let $\mathbf{NF}$ be an 8-complex with the minimal model
  $$\mathcal{M}_{\mathbf{NF}}^*=\langle x^{(3)},y^{(3)},z^{(5)},T^{(10)},\ldots \mid
  dx=dy=0, dz=xy, dT=xyz, \ldots \rangle.$$
  The geometry of this complex is discussed further in \S\ref{S:D&G}.  Here we
  focus on the algebra.  Note that $\pi_{10}(\mathbf{NF})$ has a single rational
  generator.  Suppose $f:S^{10} \to \mathbf{NF}$ is $L$-Lipschitz; then we get
  the following partial nullhomotopy of $f^*m_{\mathbf{NF}}$:
  \begin{align*}
    \Phi(x) &= f^*m_{\mathbf{NF}}x \otimes (1-t)-c(x) \otimes dt \\
    \Phi(y) &= f^*m_{\mathbf{NF}}y \otimes (1-t)-c(y) \otimes dt \\
    \Phi(z) &= f^*m_{\mathbf{NF}}z \otimes (1-t)+\frac{1}{2}(f^*m_{\mathbf{NF}}x
    \wedge c(y)-c(x) \wedge f^*m_{\mathbf{NF}}y) \otimes (t-t^2)-c(z) \otimes dt,
  \end{align*}
  where
  \begin{align*}
    dc(x) &= f^*m_{\mathbf{NF}}x \\
    dc(y) &= f^*m_{\mathbf{NF}}y \\
    dc(z) &= f^*m_{\mathbf{NF}}z+\frac{1}{2}(f^*m_{\mathbf{NF}}x \wedge c(y)-
    c(x) \wedge f^*m_{\mathbf{NF}}y).
  \end{align*}
  Then our algorithm computes the obstruction to extending the nullhomotopy to
  $T$ as
  $$-\frac{1}{3}(c(x) \wedge f^*m_{\mathbf{NF}}(y \wedge z)+c(y) \wedge
  f^*m_{\mathbf{NF}}(z \wedge x)+c(z) \wedge f^*m_{\mathbf{NF}}(x \wedge y)).$$
  Now, $\lVert c(x) \rVert_\infty$ and $\lVert c(y) \rVert_\infty \lesssim L^3$ by
  Lemma \ref{lem:IP}, but the same argument only yields
  $\lVert c(z) \rVert \lesssim L^6$.  This gives a bound of $O(L^{11})$ for the
  first two terms but $O(L^{12})$ for the last.  On the other hand, the last term
  can be eliminated by subtracting the exact form
  $d(c(z) \wedge f^*m_{\mathbf{NF}}z)$.  Thus we get an overall bound
  $\langle T,f \rangle=O(L^{11})$.  As we will see below, this bound is sharp.
\end{ex}

\section{The shadowing principle}

\begin{thm}[The shadowing principle] \label{thm:main}
  Let $(X,A)$ be an $n$-dimensional simplicial pair with the standard metric on
  simplices and $Y$ a simply connected compact Riemannian manifold with boundary
  which has a minimal model $m_Y:\mathcal{M}_Y^* \to \Omega^*Y$.  Fix norms on
  the spaces $V_k$ of $k$-dimensional indecomposables of $\mathcal{M}_Y^*$.  Let
  $f:X \to Y$ be a map and $\ph:\mathcal{M}_Y^* \to \Omega^*X$ a homomorphism
  such that
  \begin{enumerate}
  \item $f^*m_Y|_A=\ph|_A$ (i.e.~the homomorphisms restrict to the same
    homomorphism $\mathcal{M}_Y^* \to \Omega^*A$.)
  \item $f^*m_Y$ and $\ph$ are homotopic rel $A$ (i.e.~via a homotopy whose
    restriction to $A$ is constant.)
  \item $f|_A$ is $L$-Lipschitz.
  \item $\Dil(\ph) \leq L$.
  \end{enumerate}
  Then $f$ is homotopic rel $A$ to a $C(L+1)$-Lipschitz map $g:X \to Y$ such that
  $g^*m_Y \simeq \ph$ via a homotopy
  $$\Phi:\mathcal{M}^*_Y \to \Omega^*(X) \otimes \langle t,dt \rangle$$
  whose restriction to $A$ is constant, such that $\Dil_{1/L}(\Phi) \leq C(L+1)$.
  The constant $C$ depends on $Y$, $m_Y$, and the norms on indecomposables, as
  well as $n$ (but not anything else about $X$.)
\end{thm}
The condition that $Y$ be a manifold is only necessary for the technical
definitions.  In most applications, one can use any space which is Lipschitz
homotopy equivalent to a manifold, for example any simplicial complex with a
piecewise linear metric.

As stated in the introduction, we want to interpret the shadowing principle as
saying that pullbacks of genuine maps have reasonably high density in
$\Hom(\mathcal{M}_Y^*,\Omega^*X)$ when it is endowed with a metric of the form
$$d(\ph,\psi)=
\inf \bigl\{\text{size}(\Phi): \ph \overset{\Phi}{\simeq} \psi\bigr\},$$
for some notion of size.  However, there is some difficulty in defining an
appropriate such notion---that is, we would like the size of a constant homotopy
to be zero and the notion of distance to be nondegenerate and satisfy the
triangle inequality, and this is already nontrivial.  One notion that satisfies
these two properties, at least when the source is a minimal model, is the
\emph{formal length}, given by
$$\lengt(\Phi)=\Dil\bigl({\textstyle \int_0^1} \Phi\bigr).$$
The triangle inequality is given by Lemma \ref{prop:cat} and nondegeneracy
follows from applying \eqref{eqn:I01} to the lowest-degree indecomposable on
which the two homomorphisms differ.

Under this metric, the theorem states that there is a pullback of a genuine map
within distance $O(L)$ of any homomorphism with dilatation $L$ which lies in the
homotopy class of the pullback of a genuine map.  Put this way, this is a
nontrivial statement since the set of all homomorphisms with dilatation $\leq L$
has diameter which is in general some polynomial in $L$; this polynomial is
linear only when $Y$ has finite homotopy groups up to dimension $n$.

Unfortunately, the formal length does not correspond well to the length of a
genuine homotopy, or, as far as I can tell, any geometric invariant of genuine
maps.  In particular, as shown in \cite{CaSi} and again in this paper in Theorem
\ref{thm:htpy}, when the space $\Map(X,Y)$ is equipped with the metric given by
the optimal (geometric) length of a homotopy (ignoring thickness), the diameter
of each connected component is finite, with a uniform bound.
\begin{proof}
  By subdividing $(X,A)$ at scale $1/L$ and rescaling so that simplices are unit
  size, we may assume $L=1$; here we implicitly use the uniformity of the result
  with respect to the large-scale geometry of $X$.  We also subdivide once if the
  star of $A$, denoted $\st(A)$, does not retract to $A$.

  At the cost of increasing the Lipschitz constant again to some $C_0=C_0(n,Y)$,
  we may also assume that $f|_A$ is mosaic with respect to a fixed shard complex
  $Z \subset \Delta Y$ with $Z^{(1)}=*$.  To reduce to this case, we modify both
  $f$ and $\ph$ on $\st(A)$, which we equip with a facewise linear deformation
  retraction to $A$,
  $$\pi:\st(A) \times [0,1] \to \st(A),$$
  and a simplexwise linear map $\tau:\st(A) \to [0,1]$ sending $A \mapsto 0$ and
  $\lk(A) \mapsto 1$.  Let $H:A \times [0,1] \to Y$ be a $C(n,Y)$-Lipschitz
  homotopy to a $C(n,Y)$-Lipschitz mosaic map on some chosen shard complex.
  (Such a homotopy can be constructed by simplicially approximating on a complex
  which is homotopy equivalent to $Y$.)  We use this homotopy on a collar of
  width $1/2$ around $A$, pushing $f|_{\st A}$ to the outer part of the collar:
  $$\hat f(x)=\left\{\begin{array}{l l}
  H(\pi(x,1),1-2\tau(x)) & \tau(x) \leq 1/2 \\
  f(\pi(x,2-2\tau(x)) & \tau(x) \geq 1/2.
  \end{array}\right.$$
  We push $\ph$ to the outer half of $\st(A)$ by a similar formula, adding
  $H^*m_Y$ on the inner half; this gives us an algebraic map $\hat\ph$.  Applying
  the rest of the proof to $\hat f$ and $\hat\ph$, we produce a map $\hat g$ with
  the required properties such that $\hat g|_A=\hat f|_A$.  To get the desired
  $g$ we again push $\hat g|_{\st(A)}$ to the outer 2/3 of the star and add $H$,
  going in the opposite direction, to the collar.  To show that the resulting $g$
  indeed has a short homotopy to $\ph$, note that it clearly has a short homotopy
  to the algebra map $\hat{\hat\ph}$ which is given by $H^*m_Y$ on the inner
  third of $\st(A)$, pushing $\hat\ph$ out.  But this map in turn has a short
  homotopy to $\ph$.

  We now give an overview of the induction on skeleta that characterizes the rest
  of the proof.  At the $(k+1)$st step, we will produce an increasingly
  controlled intermediate map $g_{k+1}$ which is homotopic to $g_k$ via a homotopy
  $H_{k+1}$ (and therefore homotopic to $f$).  In particular, $g_k$ will be equal
  to the final $g$ on the $k$-skeleton of $X$ and $H_{k+1}$ will be a constant
  homotopy on the $(k-1)$-skeleton; its behavior on $k$-cells is crucial for
  establishing control over the behavior of $g_{k+1}$ on $(k+1)$-simplices.
  Essentially, the behavior of $g_k$ on $(k+1)$-simplices allows us to define an
  ``almost coboundary'' in $C^{k+1}(X;\pi_{k+1}(Y) \otimes \mathbb{R})$ and the
  homotopy $H_{k+1}$ changes this cochain by the coboundary that it almost is,
  leaving a uniformly bounded remainder.

  In order to figure out a recipe for doing this which can be continued further,
  we consult a homotopy $\Phi_k$ between $g_k$ and $\ph$ over which we also have
  increasing control depending on $k$.  We then construct $\Phi_{k+1}$ from
  $\Phi_k$ and $H_{k+1}$ via a second-order homotopy $\Psi_{k+1}$.  The objects we
  produce are summarized in Figure \ref{fig}.
  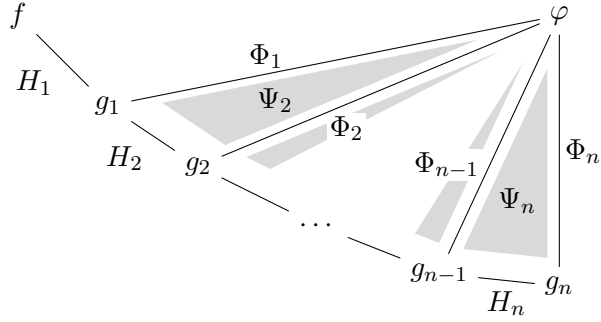
\begin{figure}
    \centering
    \begin{tikzpicture}[scale=1.2]
      \node (phi) at (0,0) {$\ph$};
      \node (g0) at (-6,0) {$f$};
      \node (g1) at (-5,-1) {$g_1$};
      \node (g2) at (-4,-1.67) {$g_2$};
      \node (vd) at (-2.67,-2.33) {$\cdots$};
      \node (gn-1) at (-1.33,-2.86) {$g_{n-1}$};
      \node (gn) at (0,-3) {$g_n$};
      \fill[color=gray!30!white] (barycentric cs:phi=0.8,g1=0.1,g2=0.1) --
      (barycentric cs:phi=0.1,g1=0.8,g2=0.1) --
      (barycentric cs:phi=0.1,g1=0.1,g2=0.8);
      \fill[color=gray!30!white] (barycentric cs:phi=0.8,g2=0.1,vd=0.1) --
      (barycentric cs:phi=0.1,g2=0.8,vd=0.1) --
      (barycentric cs:phi=0.1,g2=0.6,vd=0.3);
      \fill[color=gray!30!white] (barycentric cs:phi=0.8,gn-1=0.1,vd=0.1) --
      (barycentric cs:phi=0.1,gn-1=0.8,vd=0.1) --
      (barycentric cs:phi=0.1,gn-1=0.6,vd=0.3);
      \fill[color=gray!30!white] (barycentric cs:phi=0.8,gn-1=0.1,gn=0.1) --
      (barycentric cs:phi=0.1,gn-1=0.8,gn=0.1) --
      (barycentric cs:phi=0.1,gn-1=0.1,gn=0.8);
      \draw (g1) -- (phi) node[pos=0.33,anchor=south,inner sep=2pt] {$\Phi_1$}; 
      \draw (g2) -- (phi) node[pos=0.33,fill=white,anchor=north west,inner sep=1pt] {$\Phi_2$};
      \draw (gn-1) -- (phi) node[pos=0.33,fill=white,anchor=south east,inner sep=1pt] {$\Phi_{n-1}$};
      \draw (gn) -- (phi) node[midway,anchor=west,inner sep=2pt] {$\Phi_n$};
      \draw (g0) -- (g1) node[midway,anchor=north east] {$H_1$};
      \draw (g1) -- (g2) node[midway,anchor=north east] {$H_2$};
      \draw (g2) -- (vd) (vd) -- (gn-1);
      \draw (gn-1) -- (gn) node[midway,anchor=north] {$H_n$};
      \node at (barycentric cs:phi=0.3,g1=0.35,g2=0.35) {$\Psi_2$};
      \node at (barycentric cs:phi=0.3,gn-1=0.35,gn=0.35) {$\Psi_n$};
    \end{tikzpicture}
    \caption{\label{fig} A summary of the various maps, homotopies, and second
      order homotopies produced in the proof of Theorem \ref{thm:main}.  The
      bottom row consists of genuine maps $X \to Y$ and homotopies between them;
      the rest is on the level of DGAs.  Maps become better controlled from left
      to right.}
  \end{figure}

  As a first step, we homotope $f$ rel $A$ to a map $g_1$ which sends $X^{(1)}$ to
  the basepoint of $Y$.  We also choose a homotopy
  $$\Phi_1:\mathcal{M}_Y^* \to \Omega^*(X) \otimes
  \mathbb{R}\langle t,dt \rangle$$
  between $g_1^*m_Y$ and $\ph$.

  After the $k$th step, we assume that we have constructed the following:
  \begin{itemize}
  \item a map $g_k:(X,A) \to Y$, homotopic rel $A$ to $f$, such that
    $g_k|_{X^{(k)}}$ is mosaic with respect to a shard complex
    $Z_k \subset \Delta Y$ which depends only on $Y$, $m_Y$, and the norms on the
    $V_i$;
  \item a homotopy $\Phi_k:\mathcal{M}_Y^* \to \Omega^*(X) \otimes
    \mathbb{R}\langle t,dt \rangle$ from $g_k^*m_Y$ to $\ph$ such that
    $$\Dil_1\bigl((\Phi_k|_{\mathcal{M}^*_Y(k)})|_{X^{(k)}}\bigr) \leq C_k=C_k(n,Y).$$
  \end{itemize}
  We write $\beta_k=\int_0^1 \Phi_k$; note that for $v \in V_i$,
  $d\beta_k(v)=\ph(v)-g_k^*m_Y(v)-\int_0^1 \Phi_k(dv)$ and $\beta_k(v)|_A=0$.

  We then construct the analogues one dimension higher.  Let
  $b \in C^k(X,A;\pi_{k+1}(Y))$ be the simplicial cochain obtained by
  integrating $\beta_k|_{V_{k+1}}$ over $k$-simplices and choosing an element of
  $\pi_{k+1}(Y)$ whose image in $V_{k+1}$ is as close as possible in norm (but
  which is otherwise arbitrary.)  Note that the values of $b$ are potentially
  unbounded.  We use $b$ to specify a homotopy $H_{k+1}:X \times [0,1] \to Y$
  from $g_k$ to our new desired map $g_{k+1}$.

  We start by setting $H_{k+1}$ to be constant on $X^{(k-1)}$.  On each $k$-simplex
  $q$, we set $H_{k+1}|_q$ to be a map such that
  $$g_{k+1}|_q=H_{k+1}|_{q \times \{1\}}=H_{k+1}|_{q \times \{0\}}=g_k|_q,$$
  but such that on the cell $q \times [0,1]$, the map traces out the element
  $\langle b, q \rangle \in \pi_{k+1}(Y)$.  This is well-defined since
  $H_{k+1}|_{\partial(q \times [0,1])}$ is canonically nullhomotopic by precomposition
  with a linear contraction of the simplex.

  Now, given that $g_{k+1}=g_k$ on the $k$-skeleton, the possible relative
  homotopy classes of the restriction of $g_{k+1}$ to a $(k+1)$-simplex $p$ form
  a torsor for $\pi_{k+1}(Y)$.  No matter how we extend $H_{k+1}$ over
  $p \times [0,1]$, we will get $g_{k+1}|_p-g_k|_p=\langle \delta b, p \rangle$ in
  this torsor.  For each possible restriction $g_k|_{\partial p}$ (of which there
  are finitely many since they correspond to simplicial maps
  $\partial\Delta^{k+1} \to Z_k$) we fix representatives for each element of this
  torsor.  We then set $g_{k+1}|_p$ to be the appropriate representative.

  We then extend the homotopy in an arbitrary way to higher skeleta.

  We now argue that, for a given $Z_k$-mosaic map
  $u_0:\partial \Delta^{k+1} \to Y$, the number of extensions of $u_0$ to
  $\Delta^{k+1}$ which could occur as $g_{k+1}|_p$ for some $(k+1)$-simplex $p$ are
  drawn from a finite set depending only on $C_k$, $Y$, and the norms on the
  $V_i$, $i \leq k+1$.  At various stages we will write ``$\lesssim 1$''
  for numbers that are bounded by a constant depending on these items.  Thus for
  example, every such $u_0$ has a $\lesssim 1$-Lipschitz extension
  $u:\Delta^{k+1} \to Y$.  In this language, it is enough to show the following
  lemma:
  \begin{lem} \label{lem:close}
    Let $p$ be a simplex of $X$ such that $g_k|_{\partial p}=u_0$.  The homotopy
    class of the map $\tilde u:S^{k+1} \to Y$ given by $g_k|_p$ on the northern
    hemisphere and the fixed extension $u$ on the southern hemisphere is
    contained in a $\lesssim 1$-ball around $\langle \delta b, p \rangle$ in
    $V_{k+1}^*$.
  \end{lem}
  Therefore, the homotopy class of the map obtained by gluing together
  $g_{k+1}|_p$ and $u$ in a similar fashion is contained in a $\lesssim 1$-ball
  around $0 \in V_{k+1}^*$.  But since $\pi_{k+1}(Y) \to V_{k+1}^*$ is a
  homomorphism from a finitely generated group whose kernel is torsion, there are
  finitely many elements in this ball.
  \begin{proof}[Proof of the lemma.]
    In \S\ref{S:pi_k}, we described the real homotopy class of
    $\tilde u:S^{k+1} \to Y$ as the obstruction in $V_{k+1}^*$ to homotoping
    $\tilde u^*m_Y$ to zero.  But equivalently, it is the obstruction to
    homotoping it to any other algebraically nullhomotopic map, for example the
    map
    $\begin{smallmatrix}\text{\rotatebox[origin=c]{180}{$\ph$}}\\
      \ph\end{smallmatrix}$
    which restricts to $\ph|_p$ on each hemisphere.

    So we build such a homotopy $\Psi$ through degree $k$, then evaluate the
    obstruction to extending it to $V_{k+1}$.  On the northern hemisphere, we
    simply use $\Psi=\Phi_k|_p$.  On the southern hemisphere, since
    $\Dil_1(\Phi_k|_{\partial p}) \lesssim 1$, we can use Prop.~\ref{prop:extHtpy}
    to make sure $\Dil_1(\Psi) \lesssim 1$.

    Now the obstruction to extending to $V_{k+1}$ is given, according to
    Prop.~\ref{prop:extBoth}, by
    $$\bigl[-\begin{smallmatrix}\text{\rotatebox[origin=c]{180}{$\ph$}}\\
      \ph\end{smallmatrix}|_{V_{k+1}}+\tilde u^*m_Y|_{V_{k+1}}
      +{\textstyle\int_0^1} \Psi d|_{V_{k+1}}\bigr] \in V_{k+1}^*.$$
    Analyzing this form separately on each hemisphere, we get that this class is
    the sum of the class sending $v \in V_{k+1}$ to
    $$\int_p\bigl(\ph(v)-g_k^*m_Y(v)-{\textstyle \int_0^1 \Phi_k(dv)}\bigr)=
    \int_p d\beta_k(v)$$
    on the northern hemisphere and a $\lesssim 1$ error coming from the southern
    hemisphere.  Thus by Stokes' theorem, it is within $\lesssim 1$ of
    $\langle \delta b, p \rangle$.
  \end{proof}

  This allows us to fix a new shard complex
  $Z_{k+1}=Z_k \cup Z^{(k+1)} \cup \mathcal{F}_{k+1}$, where $\mathcal{F}_{k+1}$ is
  the finite set of restrictions to $(k+1)$-cells we have produced.

  It remains to define the homotopy $\Phi_{k+1}$.  We do this by applying a
  restriction to a second-order homotopy.  Let $\pi:X \times [0,1] \to X$ be the
  obvious projection.  Then we will construct a homotopy
  $$\Psi_{k+1}:\mathcal{M}_Y^* \to \Omega^*(X \times [0,1]) \otimes
  \mathbb{R}\langle s,ds \rangle$$
  between $H_{k+1}$ and $\pi^*\ph$ such that $\Psi_{k+1}|_{t=0}=\Phi_k$ and such
  that $\Phi_{k+1}:=\Psi_{k+1}|_{t=1}$ has the properties we desire.  Accordingly,
  we will use the notations $\Phi_{k+1}$ and $\Psi_{k+1}|_{t=1}$ interchangeably.

  We build this homotopy by induction on the degree of indecomposables of
  $\mathcal{M}_Y^*$.  The crucial step is in degree $k+1$, since this is where we
  do not yet have control but need to establish it; thus we split the
  construction into ``before'', ``during'', and ``after''.

  \subsubsection*{For $v \in V_i$, $i \leq k$}
  In low degrees, we further induct on skeleta.  First we set
  \begin{align*}
    \Psi_{k+1}(v)|_{X \times \{0\}} &= \Phi_k(v) \\
    \Psi_{k+1}(v)|_{X^{(k-1)} \times [0,1]} &= \pi^*\Phi_k(v)|_{X^{(k-1)}} \\
    \Psi_{k+1}(v)|_{X^{(k)} \times \{1\}} &= \Phi_k(v).
  \end{align*}

  Over cells of the form $q \times [0,1]$ where $q$ is a $k$-simplex of $X$, we
  can extend in an arbitrary way by the usual Poincar\'e lemma.

  Over cells of the form $p \times \{1\}$ where $p$ is a $(k+1)$-simplex, we have
  promised to control the size of the homotopy, which is part of $\Phi_{k+1}$.
  Recall that $\beta_k=\int_0^1 \Phi_k$.  For $v \in V_i$, $i \leq k$, we let
  $$\Phi_{k+1}(v)=g_{k+1}^*m_Y(v)
  +d(\beta_{k+1}(v) \otimes s)+{\textstyle \int_0^s \Phi_{k+1}(dv)};$$
  we would like to define such a $\beta_{k+1}$ on $V_i$ which is bounded on
  $X^{(k+1)}$ and such that
  $$d\beta_{k+1}(v)=\ph(v)-g_{k+1}^*m_Y(v)-{\textstyle \int_0^1 \Phi_{k+1}(dv)}.$$
  By induction, $d\beta_{k+1}=d\beta_k$ on $X^{(k)}$ and
  $$\lVert d\beta_{k+1}|_{X^{(k+1)}} \rVert_{\mathrm{op}} \lesssim 1.$$
  By the second quantitative Poincar\'e lemma, we can therefore extend
  $\beta_k|_{X^{(k)}}$ to a $\beta_{k+1}|_{X^{(k+1)}}$ with
  $$\lVert \beta_{k+1}|_{X^{(k+1)}} \rVert_{\mathrm{op}} \lesssim 1.$$
  Since $i \leq k$ all such choices differ by coboundaries.

  On all higher cells, we can once again extend in an arbitrary way by the usual
  Poincar\'e lemma.

  \subsubsection*{For $v \in V_{k+1}$}
  We need to ensure that $\Phi_{k+1}|_{V_{k+1}}$ has low dilatation on $(k+1)$-cells
  of $X$.  As before, we will set
  $$\Phi_{k+1}(v)=g_{k+1}^*m_Y(v)
  +d(\beta_{k+1}(v) \otimes s)+{\textstyle \int_0^s \Phi_{k+1}(dv)}$$
  where $d\beta_{k+1}(v)=\ph(v)-g_{k+1}^*m_Y(v)-{\textstyle\int_0^1\Phi_{k+1}(dv)}$.
  Specifically, we determine $\beta_{k+1}$ as follows:
  \begin{itemize}
  \item Take $\beta_{k+1}|_q$ to be the volume form times a bump function scaled
    so that
    $$\int_q \beta_{k+1}(v)=\int_q \beta_k(v)-\langle b,q \rangle(v).$$
  \item Use the second quantitative Poincar\'e lemma to extend $\beta_{k+1}(v)$
    to $p \times \{1\}$, for every $(k+1)$-simplex $p$, such that
    $d\beta_{k+1}(v)$ is as desired and
    $\lVert \beta_{k+1}|_{X^{(k+1)}} \rVert_{\mathrm{op}} \lesssim 1$.
  \item Extend arbitrarily to higher skeleta by the usual Poincar\'e lemma.
  \end{itemize}

  By Prop.~\ref{prop:extBoth}, the obstruction to extending this to a definition
  of $\Psi_{k+1}(v)$ is given by a class in
  $H^{k+1}(X \times [0,1],X \times \{0,1\};V)$ defined by
  $$O(v)=\bigl(\pi^*\ph(v)-H_{k+1}^*m_Y(v)-{\textstyle \int_0^1 \Psi_{k+1}(dv)},
  \beta_k(v) \oplus \beta_{k+1}(v)\bigr).$$
  In other words, we can get such an extension if there is a form
  $B(v) \in \Omega^k(X \times [0,1])$ such that
  $O(v)=(dB(v),B(v)|_{X \times \{0,1\}})$.

  In fact, we can find such a $B$ with $B(v)=0$ on the $(k-1)$-skeleton of $X$.
  By the Poincar\'e lemma, it is enough that $B$ satisfy Stokes' theorem, in
  other words that for $q$ a $k$-simplex of $X$,
  $$\int_{q \times [0,1]} dB(v)=\int_q \beta_{k+1}(v)-\int_q \beta_k(v).$$
  Therefore, we just need to show that
  $\int_{q \times [0,1]} dB(v)=-\langle b,q \rangle(v)$.

  To do this, notice that both $H_{k+1}$ and $dB(v)$ factor through the map
  $q \times [0,1] \to S^{k+1}$ which identifies $q \times \{0\}$ with
  $q \times \{1\}$ and flattens $\partial q \times [0,1]$ to $\partial q$.
  Thus $\int_{q \times [0,1]} dB$ is the obstruction to homotoping $H_{k+1}^*m_Y$
  to $\pi^*(\ph|_q)$ in this quotient; the latter is algebraically nullhomotopic
  since it factors through $\Omega^*(D^k)$.  By the construction of $H_{k+1}$,
  this is $-\langle b,q \rangle$.

  \subsubsection*{For $v \in V_i$, $i>k+1$}
  Finally, we extend to $\mathcal{M}^*_Y$ by applying the relative obstruction
  lemma Prop.~\ref{prop:extBoth} to the diagram
  $$\xymatrix{
    \mathcal{M}^*_Y(k+1) \ar[r]^-{\Psi_{k+1}} \ar@{^{(}->}[d] &
    \Omega^*(X \times [0,1]) \otimes \mathbb{R}\langle s,ds \rangle
    \ar[d]^{t=0} \ar[rd] \\
    \mathcal{M}^*_Y \ar[r]^-{\Phi_k} \ar@{-->}[ru] &
    \Omega^*X \otimes \mathbb{R}\langle s,ds \rangle \ar@{=}[r] &
    \Omega^*X \otimes \mathbb{R}\langle s,ds \rangle,
  }$$
  in which the middle vertical arrow is a quasi-isomorphism.  This completes the
  construction of $\Phi_{k+1}$ and the inductive step.  When $k=\dim X$, the
  result is the statement of the theorem.
\end{proof}

\section{Applications}

\subsection{Distortion and growth} \label{S:D&G}

In theory, our results reduce Conjectures \ref{conj:dist} and \ref{conj:growth}
to purely algebraic questions about the homotopy theory of maps between algebras
of forms.  In reality, however, it is not clear whether these questions are any
easier to answer than the geometric questions they come from.  In this section,
we give some examples of geometric constructions that confirm Conjecture
\ref{conj:dist} for certain types of spaces, as well as a first attempt at a
general theorem using our machinery.  First, however, there is the following
result, which is almost a triviality given the shadowing principle:
\begin{thm} \label{thm:triv}
  Let $X$ be an $n$-dimensional simplicial complex with the standard simplexwise
  metric and $Y$ a simply connected finite complex.  Then there is $C(n,Y)$ such
  that if $\alpha,\beta \in [X,Y]$ are homotopy classes which are the same
  rationally, then
  $\lVert\alpha\rVert_{\mathrm{Lip}} \leq C(\lVert\beta\rVert_{\mathrm{Lip}}+1)$.
\end{thm}
The remarkable aspect is that this constant does not depend on the particular
rational homotopy class or even on the topology of $X$, but only on its bounded
geometry.  This is although the \emph{number} of distinct homotopy classes within
a rational homotopy class may be unbounded, even for a fixed $X$; see
\cite{IRMC} for examples of this phenomenon.
\begin{proof}
  To apply the shadowing principle, we need $Y$ to be a Riemannian manifold with
  boundary.  So we embed $Y$ in some $\mathbb{R}^N$ and thicken it up to a
  manifold $Y^\prime$.  The map $Y \hookrightarrow Y^\prime$ is a Lipschitz
  homotopy equivalence, so this affects the Lipschitz norm of homotopy classes
  only by a multiplicative constant $C(Y \hookrightarrow Y^\prime)$.

  Let $f:X \to Y^\prime$ be a (near-)optimal representative of $\beta$ and
  $g:X \to Y^\prime$ some representative of $\alpha$.  Choose a minimal model
  $m_Y:\mathcal{M}^*_Y \to \Omega^*Y^\prime$.  The shadowing principle allows us
  to deform $g$ to a map $\tilde g$ such that $\tilde g^*m_Y$ is near $f^*m_Y$;
  in particular $\Lip(\tilde g) \leq C(n,Y,m_Y)(\Lip(f)+1)$.
\end{proof}


\subsubsection*{Universal constructions}

There are a number of cases in which the distortion of elements of homotopy
groups can be determined geometrically, without using the machinery introduced in
this paper.  Gromov originally noted in \cite{GrHED} that given a map
$f:S^{2n-1} \to S^n$ with nonzero Hopf invariant $h$ and Lipschitz constant $L$, a
map with Hopf invariant $k^{2n}h$ and Lipschitz constant $\lesssim kL$ can be
produced by composing with a self-map
$$S^{2n-1} \xrightarrow{f} S^n \xrightarrow{\deg=k^n} S^n.$$

More generally, a large number of homotopy group elements can be represented by
the following universal construction.  Given spheres $S^{n_1},\ldots,S^{n_r}$,
their product can be given a cell structure with one cell for each subset of
$\{1,\ldots,r\}$.  Define their \emph{fat wedge} $\mathbb{V}_{i=1}^r S^{n_r}$ to be
this cell structure without the top face.  Let $N=-1+\sum_{i=1}^r n_i$, and let
$\tau:S^N \to \mathbb{V}_{i=1}^r S^{n_r}$ be the attaching map of the missing face.
By definition, $\alpha \in \pi_N(Y)$ is contained in the
\emph{$r$th-order Whitehead product} $[\alpha_1,\ldots,\alpha_r]$, where
$\alpha_i \in \pi_{n_i}(Y)$, if it has a representative which factors through a
map
$$S^N \xrightarrow{\tau} \mathbb{V}_{i=1}^r S^{n_i} \xrightarrow{f_\alpha} Y$$
such that $[f_\alpha|_{S^{n_i}}]=\alpha_i$.  Note that there are many potential
indeterminacies in how higher-dimensional cells are mapped, so
$[\alpha_1,\ldots,\alpha_r]$ is a set of homotopy classes rather than a unique
class.\footnote{See \cite{AA} for the relationship between generalized Whitehead
  products and Sullivan minimal models.}

Nevertheless, as long as each of the $n_i \geq 2$, any class in this set has
distortion $O(k^{1/(N+1)})$, for the following reason.  Let
$\sigma_i:S^{n_i} \to S^{n_i}$ be an $O(L)$-Lipschitz map of degree $L^{n_i}$.  Then
the product of the $\sigma_i$'s induces a self-map of the
fat wedge which has degree $L^{N+1}$ on the missing cell.  Since the fat wedge is
simply connected, the relative Hurewicz theorem gives an isomorphism
$$\pi_{N+1}\bigl({\textstyle\prod_{i=1}^r} S^{n_r},\mathbb{V}_{i=1}^r S^{n_r}\bigr)
\xrightarrow{\simeq}
H_{N+1}\bigl({\textstyle\prod_{i=1}^r} S^{n_r},\mathbb{V}_{i=1}^r S^{n_r}\bigr).$$
Thus the composition
$$S^N \xrightarrow{\tau} \mathbb{V}_{i=1}^r S^{n_r} \xrightarrow{\prod \sigma_i}
\mathbb{V}_{i=1}^r S^{n_r} \xrightarrow{f_\alpha} Y$$
gives us an $O(L)$-Lipschitz representative of $L^{N+1}\alpha$.

This class of examples has not been described in detail before, but it is not
original to this paper.  It was mentioned by Gromov in \cite{GrQHT} and it also
provides the tools to prove the following observation of Shmuel Weinberger:
\begin{thm}
  The following are equivalent for a finite, simply connected complex $Y$:
  \begin{enumerate}[label={(\roman*)}]
  \item All rationally nontrivial elements of $\pi_*(Y)$ are undistorted.
  \item The rational Hurewicz map $\pi_*(Y) \otimes \mathbb{Q} \to
    H_*(Y;\mathbb{Q})$ is injective.
  \item $Y$ is rationally equivalent to a product of odd-dimensional spheres.
  \end{enumerate}
\end{thm}
\noindent This follows from the fact that the lowest-dimensional nonzero element
in the kernel of the rational Hurewicz map is always a generalized Whitehead
product.  This is shown in \cite[Lemma 5.2]{thesis}.

There are also more subtle examples of similar constructions.  One of the
simplest examples of a rational homotopy class which is not a generalized
Whitehead product is contained in the $\pi_{10}$ of an 8-dimensional, four-cell
CW complex constructed as follows.  Ordinary Whitehead products satisfy the
relations for a Lie bracket: in particular, they are bilinear and satisfy the
Jacobi identity.  This can be demonstrated via universal topological
constructions, as in \cite{NaTo}.  In particular, the rational homotopy groups of
$S^3 \vee S^3$ are a free Lie algebra whose Lie bracket is the Whitehead product,
generated by the identity maps of the two spheres, which we call $f$ and $g$.  So
we attach two $8$-cells killing
$$\pi_7(S^3 \vee S^3) \otimes \mathbb{Q} \cong
\langle [[f,g],f],[[f,g],g] \rangle,$$
to get a space $\mathbf{NF}$.  (This stands for ``non-formal'', as $\mathbf{NF}$
is also one of the simplest examples of a space which is not formal in the sense
of Sullivan.)  Then
$\pi_{10}(\mathbf{NF}) \otimes \mathbb{Q}$ is generated by a single element.
This can be seen by constructing the first few levels of the minimal model
$$\mathcal{M}_{\mathbf{NF}}^*=\langle x_1^{(3)},x_2^{(3)},y^{(5)},T^{(10)},\ldots \mid
dx_i=0, dy=x_1x_2, dT=x_1x_2y, \ldots \rangle,$$
but we also give an explicit generator.
\begin{lem}
  An explicit generator $\tau:S^{10} \to \mathbf{NF}$ for $\pi_{10}(\mathbf{NF})$
  is given by the following sequence of homotopies between maps
  $S^9 \to \mathbf{NF}$, coned off at both ends:
  $$* \xlongequal{} [*_{S^7},f] \xlongequal{[\text{nullh.},f]} [[[2f,g],g],f]
  \xlongequal{\text{Jacobi}} [[[2f,g],f],g] \xlongequal{[\text{nullh.},g]}
  [*_{S^7},g] \xlongequal{} *.$$
\end{lem}
Note that the Jacobi identity takes this form since $[[f,g],[f,g]]$ has order 2
for degree reasons.
\begin{proof}
  Since $\pi_{10}(\mathbf{NF})$ has rank 1, we just need to show that the given
  map pairs nontrivially with $T$.

  We do this as follows.  Let
  \begin{align*}
    Y_1 &= (S^3 \vee S^3) \cup_{[[f,g],f]} e^8 \\
    Y_2 &= (S^3 \vee S^3) \cup_{[[f,g],g]} e^8 \\
    Z &= (S^3 \vee S^3) \cup_{[[[2f,g],f],g]} e^{10},
  \end{align*}
  and let $\iota_i:Y_i \to \mathbf{NF}$, $i=1,2$, be the obvious inclusions.
  Then there is a map $f_1:Z \to Y_1$ which is the identity on the $3$-skeleton
  and sends the $10$-cell to $Y_1$ via the right hemisphere of $\tau$.
  Similarly, there is a map $f_2:Z \to Y_2$ which acts on the $10$-cell via the
  left hemisphere of $\tau$.  We would like to show that $\iota_1 \circ f_1$ and
  $\iota_2 \circ f_2$ are rationally distinct; in other words, that the two
  halves of $\tau$ are rationally non-homotopic nullhomotopies of
  $[[[2f,g],f],g]$.

  We argue via minimal models.  Through degree 8, we have
  \begin{align*}
    \mathcal{M}_{Y_1}^* &= \langle x_1^{(3)},x_2^{(3)},y^{(5)},z_2^{(7)},\ldots \mid
    dx_i=0, dy=x_1x_2, dz_2=x_2y, \ldots \rangle \\
    \mathcal{M}_{Y_2}^* &= \langle x_1^{(3)},x_2^{(3)},y^{(5)},z_1^{(7)},\ldots \mid
    dx_i=0, dy=x_1x_2, dz_1=x_1y, \ldots \rangle \\
    \mathcal{M}_Z^* &= \langle x_1^{(3)},x_2^{(3)},y^{(5)},z_1^{(7)},z_2^{(7)},\ldots
                      \mid dx_i=0, dy=x_1x_2, dz_i=x_iy, \ldots \rangle,
  \end{align*}
  with more generators in degree 9; clearly, the maps
  $f_i^*:\mathcal{M}_{Y_i}^* \to \mathcal{M}_Z^*$ must send the generators $x_i$,
  $y$, and $z_i$ to themselves.

  Likewise, the map $\iota_1^*:\mathcal{M}_{\mathbf{NF}}^* \to \mathcal{M}_{Y_1}^*$
  sends $x_i$ and $y$ to themselves.  By obstruction theory, since
  $\pi_8(\mathbf{NF})$ is finite, this determines its homotopy class; after
  making a choice within this homotopy class, we can send $T \mapsto -x_1z_2$.
  Similarly, $\iota_2^*:\mathcal{M}_{\mathbf{NF}}^* \to \mathcal{M}_{Y_2}^*$ sends
  $T \mapsto x_2z_1$.

  Now, $x_1z_2+x_2z_1$ is cohomologically nontrivial in $\mathcal{M}_Z^*$ since it
  is dual to the added $10$-cell (see \cite[\S13(d) and (e)]{FHT} for more
  detail.)  This gives a rational obstruction to homotoping the maps
  $\iota_1 \circ f_1$ and $\iota_2 \circ f_2$.
\end{proof}

To demonstrate that this element is distorted, we exhibit a representative of
$L^{11}\tau$,
\begin{multline*}
  * \xlongequal{} [*_{S^7},L^3f] \xlongequal{[\text{nullh.},L^3f]}
  [[2L^5[f,g],L^3g],L^3f] \\
  \xlongequal{\text{Jacobi}} [[2L^5[f,g],L^3f],L^3g]
  \xlongequal{[\text{nullh.},L^3g]} [*_{S^7},L^3g] \xlongequal{} *,
\end{multline*}
keeping track of the sizes of the intermediate maps and their homotopies.
Clearly all the Whitehead products have Lipschitz constant at most $L$, including
the implicit third term of the Jacobi identity, $[2L^5[f,g],[L^3f,L^3g]]$.  Since
the Jacobi identity is given by a universal construction, it can be done in
linear space and time in terms of the Lipschitz constants of the entries.  The
nullhomotopy of $[2L^5[f,g],L^3g]$ can also be done in linear space and time
using the composition
$$D^8 \xrightarrow{\alpha} S^5 \times S^3 \xrightarrow{\sigma_5 \times \sigma_3}
S^5 \times S^3 \to \mathbf{NF},$$
where $\alpha$ is the attaching map of the top cell and $\sigma_5$ and $\sigma_3$
are maps of degree $2L^5$ and $L^3$, respectively.  Likewise with the rest of the
homotopies, which are also nullhomotopies of Whitehead products.

The trickiest part is finding an $L$-Lipschitz nullhomotopy of
$[2L^5[f,g],[L^3f,L^3g]]$, the third term of the Jacobi identity.  Note that the
bilinearity of the Whitehead product is also realized by a universal
construction; that is there is an $O(\ell)$-Lipschitz homotopy realizing the
relation
$$2^6[(\ell/2)^3 f,(\ell/2)^3 g] \simeq [\ell^3 f,\ell^3 g].$$
Suppose that $L$ is a power of $4$.  Then we can apply such homotopies repeatedly
to get
\begin{multline*}
  [2L^5[f,g],[L^3f,L^3g]] \simeq 2[[L^{5/2}f,L^{5/2}g],[L^3f,L^3g]] \\
  \simeq 2^7[[L^{5/2}f,L^{5/2}g],[(L^3/2)f,(L^3/2)g]] \simeq \cdots \simeq
  2L[[L^{5/2}f,L^{5/2}g],[L^{5/2}f,L^{5/2}g]].
\end{multline*}
The total amount of time this composition takes can be expressed as a geometric
series, and therefore it is also $O(L)$.

We have demonstrated $O(k^{1/11})$-Lipschitz representatives for $k\tau$ where $k$
is a power of $2^{22}$; this is sufficient to show that $\tau$ has distortion
$O(k^{1/11})$.  The analysis of the minimal model in \S\ref{S:pi_k} shows that
this is the best one can do.

Indeed, when one looks for homotopy group elements which are not generalized
Whitehead products, such ``nullhomotopies of Whitehead products in two different
ways'' come up naturally.  It seems possible that one can build universal models
for all rational homotopy classes (that is, all ``higher rational homotopy
operations'', as in \cite{Blanc}) by an inductive application of this method.
\begin{open}
  Can one prove Conjecture \ref{conj:dist} for all spaces by applying self-maps
  and similar geometric methods to inductively built models?
\end{open}

\subsubsection*{Symmetric spaces}

Our universal constructions generalize Gromov's Hopf invariant example in one
direction; we also generalize it in another, to a more general class of spaces
that have self-maps with the right properties.
\begin{thm} \label{thm:sym}
  Suppose that the finite complex $Y$ has the rational homotopy type of a
  Riemannian symmetric space.  Then for any $\alpha \in \pi_n(Y)$, the distortion
  function is $\Theta(k^{1/(n+1)})$ if $\alpha$ is in the kernel of the Hurewicz
  map and $\Theta(k^{1/n})$ otherwise.
\end{thm}
This is part (i) of Theorem \ref{summary:sym}; part (ii) follows immediately.

Note that it is not clear whether the theorem contains any new results beyond the
previous ones.  Symmetric spaces are formal, meaning that their rational homotopy
type is determined by their cohomology.  From looking at presentations of the
rational cohomology of nearly all symmetric spaces, it appears that their
homotopy classes can always be represented as generalized Whitehead products.
Nevertheless, there is also no obvious reason why this should be the case;
certainly formality itself is not sufficient.\footnote{An example is
  $\left[(S^3 \times S^3)^{\#2} \times S^3\right]^\circ$.  This space is formal; as
  with $\mathbf{NF}$, the boundary of the puncture can be modeled via two
  different nullhomotopies of a Whitehead product, but not as a Whitehead product
  itself.}
\begin{open}
  Are all homotopy classes of symmetric spaces contained in generalized Whitehead
  product sets?  Can this be shown other than by exhaustion?
\end{open}

The proof of the theorem heavily uses the fact that symmetric spaces admit a
splitting homomorphism of algebras $H^*(Y;\mathbb{R}) \to \Omega^*(Y)$, induced
by the harmonic forms.  There has been some study of when the harmonic forms
specifically induce such a splitting for $Y$ a manifold \cite{Kot}, but besides
formality it is not clear what the requirements are for such a splitting to
exist.
\begin{open}
  Give a topological characterization of all simplicial complexes $Y$ for which
  the quotient map $\Omega^*(Y) \to H^*(Y;\mathbb{R})$ admits a splitting as a
  homomorphism of algebras.  Perhaps the homotopy groups of such spaces are
  always generated by generalized Whitehead products?
\end{open}
\begin{proof}[Proof of Theorem \ref{thm:sym}.]
  It is not hard to see that distortion is a rational homotopy invariant.
  Therefore for any given symmetric space it is enough to show the theorem for
  symmetric spaces themselves (or a compact retract, for non-compact symmetric
  spaces.)

  We use two topological properties of symmetric spaces.  First, the
  indecomposables of the minimal model of a symmetric space split as
  $W_0 \oplus W_1$ where $W_0=\ker d$ and $dW_1 \subset \bigwedge W_0$.  This is
  true for all homogeneous spaces; one gets such a model by canceling out some
  elements of $W_0$ and $W_1$ in the (non-minimal) Sullivan model constructed in
  \cite[\S15(f)]{FHT}.  Second, symmetric spaces are \emph{geometrically formal}
  \cite{Kot}; that is, products of harmonic forms are harmonic, so in particular
  there is an algebra homomorphism $H^*(Y;\mathbb{R}) \to \Omega^*(Y)$.  This
  induces a minimal model $m_Y:\mathcal{M}_Y^* \to \Omega^*(Y)$ such that
  $m_Y(y)$ is nonzero if and only if $y$ is cohomologically nontrivial.  This is
  the property we actually use.

  This property is preserved under pullback by a rational homotopy isomorphism;
  therefore for a non-compact symmetric space we can take our space $Y$ to be an
  embedded compact deformation retract.

  Let us take a splitting $W_0 \oplus W_1$ as above.  Note that there is an
  automorphism $\rho_L:\mathcal{M}_Y^* \to \mathcal{M}_Y^*$ which sends an
  indecomposable
  $$w \mapsto \left\{\begin{array}{l l}L^{\deg w}w & w \in W_0 \\
  L^{\deg w+1}w & w \in W_1.\end{array}\right.$$

  Now, suppose $\alpha \in \pi_n(Y)$ is in the kernel of the Hurewicz map, and
  let $f:S^n \to Y$ be a representative of $\alpha$.  Recall that the
  indecomposables of $\mathcal{M}_Y^*$ are naturally isomorphic to
  $\Hom(\pi_n(Y),\mathbb{R})$.  By the method of \S\ref{S:pi_k}, we build a
  homotopy
  $$\Phi:\mathcal{M}_Y^* \to \Omega^*S^n \otimes \mathbb{R}\langle t,dt \rangle$$
  from $f^*m_Y$ to a map which sends $y \mapsto 0$ when $\deg y<n$ and which
  sends the indecomposables in degree $n$ to $\mathbb{R}d\vol$ via the map
  $v \mapsto v(\alpha)d\vol$ (the double dual of $\alpha$ in
  $\Hom(\Hom(\pi_n(Y),\mathbb{R}),\mathbb{R})$.)

  Then $\Phi \circ \rho_L$ is a homotopy between the double dual of
  $L^{n+1}\alpha$ (at $t=1$) and a map $\ph_L$ (at $t=0$) whose image in degree
  $k$ has operator norm $O(L^k)$, since it sends $W_1 \to 0$.\footnote{It is here
    that this argument definitively fails for non-formal spaces: for example, for
    $\mathbf{NF}$, any model maps the element $y$ to a nonzero form since $x_iy$
    is cohomologically nontrivial.}  In other words, $\ph_L$ is in the rational
  homotopy class of $L^{n+1}\alpha$; applying the shadowing principle, we get a
  $CL$-Lipschitz map $f_L:S^n \to Y$ whose integral homotopy class is
  $L^{n+1}\alpha$.  This proves the theorem.
\end{proof}

\subsection{Lipschitz homotopies} \label{S:liphom}

The major application of the shadowing principle is in turning algebraic
homotopies into geometric ones.  Here previous geometric bounds were poor or
nonexistent, and just producing some new ones is a big result.  We produce
several new results using the same schema:
\begin{enumerate}
\item Construct an algebraic homotopy between two maps $f,g:X \to Y$, with a
  bound on dilatation determined by their Lipschitz constants.
\item Concatenate this homotopy with an algebraic self-homotopy of $g$ so that
  the result is homotopic rel ends to a genuine homotopy between $f$ and $g$.
  This may increase the dilatation by an amount depending on the homotopy class
  of $g$.
\item Finally, apply the shadowing principle to obtain a genuine homotopy.
\end{enumerate}
The latter two steps are encapsulated in the technical Theorem \ref{tech:htpy}.

This schema can prove a number of different results, depending on the bound
achieved in the first step.  The most general and easily stated such result is as
follows.
\begin{thm} \label{thm:htpy}
  Let $Y$ be a finite simply connected complex and $X$ a finite complex of
  dimension $n$.
  \begin{enumerate}[label={(\roman*)}]
  \item There are constants $C(X,Y)$ and $p(X,Y)$ such that any homotopic
    $L$-Lipschitz maps $f \sim g:X \to Y$ are homotopic via a
    $C(L+1)^p$-Lipschitz homotopy, which can in addition be taken to have length
    $C$.
  \item Moreover, any nullhomotopic $L$-Lipschitz map is nullhomotopic via a
    homotopy of length $C$ and thickness $CL^2$.  (In particular, this is true
    for general homotopies if $X$ has the rational homotopy type of a
    suspension.)
  \end{enumerate}
\end{thm}
\begin{rmks}
  \begin{enumerate}[leftmargin=*,label={(\alph*)}]
  \item Calder and Siegel \cite{CaSi} and again Ferry and Weinberger
    \cite{FWPNAS} gave proofs that constant-length homotopies can always be
    obtained in this context, but without any geometric bounds in the space
    direction.  In this sense only the simultaneous bound on thickness is new.
  \item Theorem \ref{thm:htpy} is a stronger statement than Conjecture
    \ref{conj:htpy} as given by Gromov, since Gromov did not ask for a bound on
    lengths of homotopies.  On the other hand, since we give nonlinear thickness,
    it is weaker than Conjectures 2 and 3 in \cite{CMW}.
  \item \label{7.1} Part (ii) gives an almost sharp estimate of $O(L^{2n})$ on
    the volume of the nullhomotopy: for any completely general bound on length
    and thickness, we must have
    $$(\text{length}) \cdot (\text{thickness})^n=\Omega(L^{2n-2}).$$
    This is demonstrated by a sequence of examples first given in
    \cite[\S7.1]{CMW}.  Let $X_n$ be a space constructed by attaching
    $(n+1)$-cells to $S^2 \vee S^2$ to kill
    $\pi_n(S^2 \vee S^2) \otimes \mathbb{Q}$.  Then $\pi_n(X_n)$ is finite, but
    since the generators of $\pi_n(S^2 \vee S^2) \otimes \mathbb{Q}$ have
    distortion $\sim k^{1/(2n-2)}$, we can find $L$-Lipschitz maps $S^n \to X_n$
    for which every nullhomotopy has degree $\Omega(L^{2n-2})$ over some
    $(n+1)$-cell.

    In a more refined sense, the estimate is sometimes sharp: one cannot decrease
    the degree of the thickness bound while retaining constant length.  This can
    be seen for maps $S^3 \to S^2$.  Consider a nullhomotopic,
    $\Theta(L)$-Lipschitz such map $f$ which sends a solid torus inside $S^3$ to
    $S^2$ via a map whose cross-section has degree $L^2$ and which is constant on
    the circular fibers, and sends the complementary solid torus to the south
    pole of $S^2$.  Let $C$ be a circle on the bounding torus which links
    nontrivially (hence with linking number $L^2$) with the preimage of the north
    pole.  Then any nullhomotopy of $f$ must have relative degree $L^2$ on
    $C \times [0,1]$; therefore, if its length is constant, its thickness must be
    $\Omega(L^2)$.  I would like to thank Sasha Berdnikov for pointing out this
    argument.
  \item Moreover, in the case of maps $S^4 \times S^3 \to S^4$, the results of
    \cite[\S7.2]{CMW} show that the exponent $p$ from (i) cannot be less than
    $8/3$.  Thus the bound of (ii) does not in general hold for
    non-nullhomotopies.
  \item On the other hand, the estimate (ii) can be improved in various ways if
    we know more about the rational homotopy type of $Y$.  For example, if $Y$ is
    rationally $k$-connected, we can take the degree of the thickness bound to be
    $1+1/k$.
  \item While the estimates (i) and (ii) look similar, they are actually
    different in certain crucial respects.  The nullhomotopy estimate is very
    soft, using only facts about DGAs.  On the other hand, the estimate for
    homotopies is actually false on the level of DGAs; indeed homotopies may be
    unbounded in the size of the original map.  This can already be seen for
    homomorphisms modeling maps $S^4 \times S^3 \to S^4$, in terms of minimal
    models
    $$\langle a^{(4)}, b^{(7)} \mid da=0, db=a^2 \rangle \to
    \langle x^{(3)}, y^{(4)}, z^{(7)} \mid dx=dy=0, dz=y^2\rangle.$$
    For small $\epsi>0$, the pairs of homomorphisms $a \mapsto \epsi y$,
    $b \mapsto \epsi^2 z$ and $a \mapsto \epsi y$, $b \mapsto \epsi^2z+xy$, which
    have norm bounded independent of $\epsi$, are homotopic via the homotopy
    \begin{align*}
      a &\mapsto \epsi y-(2\epsi)^{-1}x \otimes dt\\
      b &\mapsto \epsi^2z+xy \otimes t,
    \end{align*}
    whose size increases without bound as $\epsi \to 0$.  Indeed, any homotopy
    must correctly resolve the obstruction class $(2\epsi)^{-1}x$ in degree 3.
    Thus to prove the polynomial bound we need to use the integral structure of
    the set of homotopy classes between $X$ and $Y$.  Unfortunately, the explicit
    estimate on the degree of the polynomial goes out the window in the process.
  \end{enumerate}
\end{rmks}
Our next theorem replicates and generalizes the results of \cite{CDMW} and
\cite{CMW}.  To do this requires a definition of spaces with positive weights,
discussed in \cite{BMSS}.  A simply connected space $Y$ has
($\mathbb{Q}$-)\emph{positive weights} if the indecomposables of its minimal DGA
split as $U_1 \oplus U_2 \oplus \cdots \oplus U_r$ so that for every
$t \in \mathbb{Q}$ there is an automorphism $\ph_t$ sending $v \mapsto t^iv$,
$v \in U_i$.  Examples of spaces with $\mathbb{Q}$-positive weights include
formal spaces \cite{Shiga}, coformal spaces \cite{PWHT}, as well as homogeneous
spaces and other ``depth two'' spaces whose indecomposables split as
$W_1 \oplus W_2$, where $dW_1=0$ and $dW_2 \subset \bigwedge W_1$.  A nonexample
is a complex given in \cite{MT} which is constructed by attaching a 12-cell to
$S^3 \vee \mathbb{C}\mathbf{P}^2$; other, much higher-dimensional non-examples
are given in \cite{ArLu} and \cite{Am}.

For the below theorem, it is enough that $Y$ has an $(n+1)$-connected map to a
space with positive weights.  For example, the examples in Remark \ref{7.1} above
fit the bill since they have $(n+1)$-connected maps to coformal spaces.
\begin{thm} \label{thm:posw}
  Suppose $Y$ is a finite simply connected complex with positive weights equipped
  with automorphisms $\ph_t$ and $X$ is a finite complex of dimension $n$.  Then
  there are constants $C_1(n,Y)$ and $C_2(X)$ such that any nullhomotopic
  $L$-Lipschitz map $f:X \to Y$ has a nullhomotopy of length $C_1C_2(L+1)^d$ and
  thickness $C_1(L+1)$.  Here $d$ is the number of levels in a filtration of the
  indecomposables of $\mathcal{M}^*_Y(n)$,
  $$0=W_0 \subset W_1 \subset \cdots \subset
  W_d\text{ with }\bigwedge W_d=\mathcal{M}^*_Y(n),$$
  such that $dW_j \subseteq \bigwedge W_{j-1}$ and such that there is a basis for
  the indecomposables $V_k$ in each degree $k$ such that the subspaces
  $W_j \cap V_k$ for each $j$ and
  $$U_i \cap V_k=\{v \in V_k: \ph_t(v)=t^iv\}$$
  for each $i$ are generated by subbases.
\end{thm}
\begin{rmks}
  \begin{enumerate}[leftmargin=*,label={(\alph*)}]
  \item The number $d$ is bounded above by $n-1$ since we can always take $W_j$
    to consist of all the indecomposables of $\mathcal{M}_Y^*(j+1)$.  On the
    other hand, sometimes it can be much smaller: for example, for homogeneous
    spaces or any of the other depth two spaces discussed in \cite{CMW}, we can
    always choose $d=2$.
  \item It remains unclear whether such linearly thick nullhomotopies are
    achievable when $Y$ does not have positive weights.  The example of \cite{MT}
    has an extremely complicated DGA; it may be worth looking for an example
    whose Sullivan model has a simpler presentation to test whether there is an
    obstruction to linear thickness.  Nevertheless, the requirement that the
    space be neither formal nor coformal already forces a certain amount of
    complexity in the algebra.
  \item Unlike the previous theorem, this one gives a sharp asymptotic estimate
    on the volume of nullhomotopies for the examples of \cite[\S7.1]{CMW}.  On
    the other hand, as we see in Theorem \ref{thm:weird}, it is not sharp in the
    case of maps $S^m \to S^{2n}$.
  \end{enumerate}
\end{rmks}

We now state the technical result which we use to convert estimates on algebraic
homotopies to geometric ones.  In the proofs of theorems \ref{thm:htpy} and
\ref{thm:posw}, we assume the target is a compact Riemannian manifold.  As in the
proof of Theorem \ref{thm:triv}, this can be built from a general complex by
thickening; since this thickening is a Lipschitz homotopy equivalence, it changes
the result by at most some $C(Y)$.
\begin{thm} \label{tech:htpy}
  Let $Y$ be a simply connected compact Riemannian manifold with a minimal model
  $m_Y:\mathcal{M}^*_Y \to \Omega^*(Y)$ and $X$ an $n$-dimensional finite
  simplicial complex with the standard metric.  Let $f,g:X \to Y$ be homotopic
  Lipschitz maps and let
  $$\Phi:\mathcal{M}^*_Y \to X \otimes \mathbb{R}\langle t,dt \rangle$$
  be an algebraic homotopy between $f^*m_Y$ and $g^*m_Y$ with
  $\Dil_\tau(\Phi) \leq \sigma$ for some $\tau$ and $\sigma$.  Then for every
  $\alpha \in [0,1]$ there is a homotopy between $f$ and $g$ of length
  $C(L_\lambda+1)$, where
  $$L_\lambda=\sigma\tau+P(\lVert[f]\rVert_{\Lip})^\alpha$$
  and thickness $C(L_\theta+1)$, where
  $$L_\theta=\max\{\sigma,P(\lVert[f]\rVert_{\Lip})^{1-\alpha}\},$$
  and where $C$ depends on $Y$, $n$, and the minimal model and $P$ is a
  polynomial depending on $X$ and $Y$.
\end{thm}
In particular, for results pertaining to nullhomotopic maps, the terms involving
$P$ reduce to a constant depending on $X$ and $Y$.
\begin{proof}
  We would like to construct a controlled homotopy by applying the shadowing
  principle to the pair
  $$(L_\theta X \times [0,L_\lambda], L_\theta X \times \{0,L_\lambda\}),$$
  homotoping an uncontrolled homotopy between $f$ and $g$ to be close to the
  realization
  $$\rho\Phi:\mathcal{M}^*_Y \to \Omega^*(L_\theta X \times [0,L_\lambda]).$$
  If we can do this, then we're done; but for this, the uncontrolled homotopy has
  to be actually DGA homotopic to $\rho\Phi$ rel the ends of the interval.  Such
  a homotopy may not exist.

  To resolve this, we concatenate $\rho\Phi$ with an algebraic self-homotopy of
  $f^*m_Y$; this creates a new DGA map which is DGA homotopic to an honest
  homotopy.  The numbers $L_\lambda$ and $L_\theta$ are obtained by combining the
  measurements of $\Phi$ (algebraic ``thickness'' $\sigma$ and ``length''
  $\sigma\tau$) with those of this self-homotopy, which must therefore be
  reasonably small.

  To find such a small self-homotopy, we first let $H:X \times [0,1] \to Y$ be an
  uncontrolled homotopy between $f$ and $g$.  At this point, the difference
  between the relative homotopy classes of $H$ and $\rho\Phi$ may be quite large;
  to correct this, we will simultaneously concatenate $H$ with an honest
  self-homotopy of $f$ and $\rho\Phi$ with a reasonably small algebraic one so
  that the resulting homotopies are in the same relative DGA homotopy class. 

  We implement this strategy as follows.  Let $\Psi:\mathcal{M}^*_Y \to
  \Omega^*(X \times S^1)$ be a homomorphism which restricts to $\rho\Phi$ on one
  half of the circle and $H^*m_Y$ on the other.  Now let
  $\tilde\Psi:\mathcal{M}_Y^* \to \Omega^*X \otimes
  \mathbb{R}\langle e^{(1)} \rangle$ be obtained using Prop.~\ref{extRel} to lift
  through the diagram
  $$\xymatrix{
    & \Omega^*X \otimes \mathbb{R}\langle e \rangle \ar[d]^{e \mapsto d\theta}
    \ar[r]^-{e \mapsto 0} & \Omega^*X; \\
    \mathcal{M}^*_Y \ar[r]^-{\Psi} \ar@{-->}[ru]^{\tilde\Psi} &
    \Omega^*(X \times S^1) \ar[ru]_{|_{\theta=0}} &
  }$$
  such a lift always exists since the cohomology of the vertical arrow vanishes.
  This lets us define our small algebraic self-homotopy.  Let
  $$\Xi=f^*m_Y+\eta \otimes e:\mathcal{M}_Y^* \to
  \Omega^*X \otimes \mathbb{R}\langle e \rangle$$
  be a homomorphism such that
  $$[\Xi \boxplus \tilde\Psi] \in
  [\mathcal{M}_Y^*,\Omega^*(X) \otimes \mathbb{R}\langle e \rangle]_{f^*m_Y}$$
  is the rationalization of a class in $\pi_1(Y^X;f)$; let
  $F:X \times [0,1] \to Y$ be a (uncontrolled) self-homotopy of $f$ which is a
  representative of that class.  By \cite[Lemma 5.2(i)]{IRMC} and the surrounding
  discussion, we can pick $\Xi$ to be of polynomial length, i.e.~so that for each
  $k$, $\lVert\eta|_{V_k}\rVert_{\mathrm{op}} \leq P(\lVert[f]\rVert_{\Lip})$, where
  $P$ is a polynomial depending on $X$ and $Y$.  Note also that
  $\rho(f^*m_Y+\eta \otimes dt)$ has dilatation $\leq 1$ when scaled to be a map
  $$\mathcal{M}_Y^* \to \Omega^*\left(
  \max\{\Lip f,P(\lVert[f]\rVert_{\Lip})^{1-\alpha}\}X
  \times [0,P(\lVert[f]\rVert_{\Lip})^\alpha]\right),$$
  for any $\alpha \in [0,1]$.  Now we define
  \begin{itemize}
  \item $\tilde\Phi:\mathcal{M}_Y^* \to \Omega^*(L_\theta X \times [0,L_\lambda])$
    to be the homomorphism obtained by concatenating appropriately scaled
    versions of $\rho(f^*m_Y+\eta \otimes t)$ and $\rho\Phi$;
  \item $\tilde H:L_\theta X \times [0,L_\lambda] \to Y$ to be the map obtained by
    concatenating scaled versions of $F$ and $H$.
  \end{itemize}
  Now $\tilde\Phi$ and $\tilde H^*m_Y$ are homotopic rel ends because
  concatenating either of them on both sides with $H^*m_Y$ creates a map
  $\mathcal{M}_Y^* \to \Omega^*(X \times S^1)$ which is homotopic rel
  $X \times \{\theta=0\}$ to $\rho(\Xi \boxplus \tilde\Psi)$.  Moreover,
  $\Dil(\tilde\Phi) \leq 1$.  Applying the shadowing principle to $\tilde\Phi$
  and $\tilde H$ therefore gives the controlled homotopy we desire.
\end{proof}
We now prove Theorems \ref{thm:htpy} and \ref{thm:posw}.
\begin{proof}[Proof of Theorem \ref{thm:htpy}.]
  We first handle the case of nullhomotopies.  In this situation, we will see
  that we are in the case of \eqref{eqn:better}: that is, that we can construct
  the homotopy by extending formally in each degree, without encountering a
  nontrivial obstruction.  To show this, notice first that the map
  $$[\mathcal{M}_Y^*(k+1),\Omega^*X \otimes \langle e \rangle]_0 \to
  [\mathcal{M}_Y^*(k),\Omega^*X \otimes \langle e \rangle]_0$$
  is a surjection.  This can be seen for example as follows.  Recall that
  representatives of elements of the latter group can be represented as
  $0+\eta \otimes e$ for some $\eta$; moreover, the derivation law
  \eqref{Leibniz} implies that $\eta(v)=0$ unless $v$ is indecomposable.
  Therefore the obstruction map
  $$\mathcal{O}:0+\eta \otimes e \mapsto \eta|_{dV_{k+1}} \in H^{k+1}(X;V_{k+1})$$
  is zero, and therefore the previous step in the exact sequence is a surjection.

  This means that maps to $\Omega^*X \otimes \langle e \rangle$ which are zero at
  the basepoint can be extended without obstruction.  To show that the same holds
  for nullhomotopies
  $$\mathcal{M}_Y^*(k) \to \Omega^*X \otimes \langle t,dt \rangle,$$
  fix some uncontrolled algebraic nullhomotopy
  $\Psi:\mathcal{M}_Y^* \to \Omega^*X \otimes \langle t,dt \rangle$ of $f^*m_Y$
  and let $\Phi:\mathcal{M}_Y^*(k) \to \Omega^*X \otimes \langle t,dt \rangle$ be
  the partial nullhomotopy which we would like to extend.  We perform the formal
  equivalent of joining the homotopies at both ends to form a map
  $X \times S^1 \to Y$.  Namely, let $\Xi$ be the algebraic concatenation of
  $\Psi|_{\mathcal{M}_Y^*(k)}$ and $\Phi$, as constructed in Prop.~\ref{prop:cat},
  which is a homotopy of the zero map to itself.  In particular, the image of
  $\Xi$ lies in the subalgebra
  $$\Omega^*X \otimes K \subset \Omega^*X \otimes \langle t,dt \rangle,$$
  where $K$ is the set of elements $\sum_i a_it^i+b_it^idt$ for which
  $\sum_{i \geq 1} a_i=0$.  It is easy to see that $\langle dt \rangle \to K$ is an
  isomorphism on cohomology.  Thus by Prop.~\ref{extRel} we can find a homotopy
  lift
  $$\xymatrix{
    & \Omega^*X \otimes \langle e \rangle \ar[d]^{e \mapsto dt}
    \ar[r]^-{e \mapsto 0} & \Omega^*X, \\
    \mathcal{M}^*_Y(k) \ar[r]^-{\Xi} \ar@{-->}[ru] &
    \Omega^*(X) \otimes K \ar[ru] &
  }$$
  which extends without obstruction to $\mathcal{M}_Y^*(k+1)$ by the argument
  above.  Therefore, again by Prop.~\ref{extRel}, $\Xi$ also extends without
  obstruction.  Finally, applying Prop.~\ref{extRel} to the diagram
  $$\xymatrix{
    \mathcal{M}_Y^*(k) \ar[r]^-{\bar\Xi} \ar@{^{(}->}[d]
    & \Omega^*X \otimes \langle t,dt,s,ds \rangle \ar[d] \ar[r]
    & \Omega^*X \otimes \langle t \rangle/((t-1)t), \\
    \mathcal{M}_Y^*(k+1)\langle V \rangle \ar[r]^-{\text{``}\Xi-\Psi\text{''}}
    & \frac{\Omega^*X \otimes \langle t,dt,s,ds \rangle}
            {((t-1)(s-t),(t-1)(ds-dt),(s-t)dt)}
    \ar[ru]
  }$$
  we produce an extension of $\Phi$, which can therefore also be extended
  without obstruction.

  So let $f:X \to Y$ be a nullhomotopic $L$-Lipschitz map.  To choose a
  nullhomotopy $\Phi$ of $f^*m_Y$, we can use the following simple procedure. Set
  $$\Phi:\mathbb{R}=\mathcal{M}_Y^*(1) \to \Omega^*(X) \otimes
  \mathbb{R}\langle t,dt\rangle$$
  to be the trivial map.  Then at the $(k+1)$st stage, since there is no
  obstruction, we choose an extension as in Proposition
  \ref{prop:extHtpy}\ref*{num:ext} to extend to $\mathcal{M}_Y^*(k+1)$.  By
  \eqref{eqn:better}, we get $\Dil_{L^{-2}}(\Phi) \leq L^2$.  By plugging this into
  Theorem \ref{tech:htpy}, using $\alpha=0$, we get an $L^2$-Lipschitz
  nullhomotopy of $f$ of constant length.

  Now let $f \simeq g:X \to Y$ be non-nullhomotopic maps.  In this general case,
  we still construct a homotopy
  $$\Phi:\mathcal{M}_Y^* \to \Omega^*X \otimes \mathbb{R}\langle t,dt \rangle$$
  by lifting inductively from $k$ to $k+1$.  However, now the lift may be
  obstructed, so we will need to fix our partially-constructed homotopy
  $$\Phi_k:\mathcal{M}_Y^*(k) \to \Omega^*X \otimes
  \mathbb{R}\langle t,dt \rangle$$
  using a self-homotopy of $f^*m_Y$.  To get a bound on $\Phi_{k+1}$ in terms of
  $\Phi_k$, we will use the following roadmap:
  \begin{enumerate}[leftmargin=*]
  \item Using Proposition \ref{prop:extHtpy}\ref*{num:obst}, estimate the size of
    the obstruction class in $H^{k+1}(\Omega^*X;V_{k+1})$ to extending the
    homotopy.
  \item Find an element of
    $[\mathcal{M}_Y^*(k),\Omega^*X \otimes \mathbb{R}\langle e \rangle]_{f^*m_Y}$
    which maps to this obstruction class, with an estimate on the size of a
    representative $\Psi_k$.
  \item Algebraically concatenate the two homotopies; an estimate on the size of
    the new homotopy $\Phi_k^\prime$ is provided by Proposition \ref{prop:cat}.
  \item Finally, by Proposition \ref{prop:extHtpy}\ref*{num:ext}, $\Phi_k^\prime$
    lifts in a quantitative way to
    $$\Phi_{k+1}:\mathcal{M}_Y^*(k+1) \to
    \Omega^*X \otimes \mathbb{R}\langle t,dt \rangle.$$
    This provides an estimate on the dilatation of $\Phi_{k+1}$ in terms of that
    of $\Phi_k^\prime$.
  \end{enumerate}
  Together, all these polynomial estimates will give a $p$ such that
  $\Dil_{L^{-p}} \Phi_k \leq L^p$.  From there we can again apply Theorem
  \ref{tech:htpy} to obtain a genuine nullhomotopy with the given bound.

  We have explicit bounds on the degree to which steps (1), (3), and (4) distort
  the size of the homotopy.  Step (2) is where we must use the fact that $f$ and
  $g$ are honest maps between spaces, and where we lose this explicit bound on
  the degree of the polynomial.  The map
  $$\mathcal{O}: [\mathcal{M}_Y^*(k),
    \Omega^*X \otimes \mathbb{R}\langle e \rangle]_{f^*m_Y} \to H^{k+1}(X,V_{k+1})$$
  is given by $[f^*m_Y+\eta \otimes e] \mapsto [\eta|_{dV_{k+1}}]$.  For a given
  homotopy class of $f$, this is a linear map in terms of the values of $\eta$ on
  indecomposables.  Moreover, if we fix a basis on the indecomposables and on
  $H^{k+1}(X,V_{k+1})$, the resulting matrix has entries given by polynomials in
  the values of $f^*m_Y$ on indecomposables.  Given a class $c$ in the image of
  this matrix, we would like to find a bound on the minimal size of a preimage in
  terms of $\lVert c \rVert$.

  Now, as observed by Sullivan, these groups are the result of tensoring
  $$\pi_1\bigl((Y_{(k)})^X,f\bigr) \to H^{k+1}(X;\pi_{k+1}(Y))$$
  with $\mathbb{R}$.  Indeed, following \cite[\S5]{IRMC}, we can say somewhat
  more.  If we equip
  $[\mathcal{M}_Y^*(k),\Omega^*X \otimes \mathbb{R}\langle e \rangle]_{f^*m_Y}$
  with the norm which assigns to $\ph+\eta \otimes e$ the operator norm of $\eta$
  restricted to the indecomposables of $\mathcal{M}_Y^*(k)$, then the map
  $$\pi_1\bigl((Y_{(k)})^X,f\bigr) \to
  [\mathcal{M}_Y^*(k),\Omega^*X \otimes \mathbb{R}\langle e \rangle]_{f^*m_Y}$$
  is \emph{$P_k(\Lip f)$-surjective} for a polynomial $P_k$ depending on $X$ and
  $Y$, i.e.~every point in the codomain is at most distance $P_k(\Lip f)$ away
  from the image.  In particular, this gives a basis $\mathbf{b}$ for the image
  lattice whose vectors are polynomially bounded in terms of $\Lip f$; we also
  have a polynomial bound on the vectors of $\mathcal{O}(\mathbf{b})$.  Since the
  lattice is mapped to the image of $H^{k+1}(X;\pi_{k+1}(Y))$ in
  $H^{k+1}(X;V_{k+1})$, we also have a fixed lower bound (independent of $f$) on
  the $\dim(\img\mathcal{O})$-dimensional volume of the parallelotope spanned by
  $\mathcal{O}(\mathbf{b})$.  This gives a polynomial lower bound on the shortest
  axis of this parallelotope.  This completes step (2) and the proof.
\end{proof}
\begin{proof}[Proof of Theorem \ref{thm:posw}.]
  For spaces with positive weights, we can take advantage of automorphisms to
  give an alternate construction of an algebraic nullhomotopy.  Fix a map
  $f:X \to Y$ and a family of automorphisms $\ph_t:\mathcal{M}_Y^* \to
  \mathcal{M}_Y^*$ with the desired properties.  Now, let $v \in V_k \cap U_i$ be
  an element of the basis mentioned in the statement of the theorem.  Then we
  inductively define
  $$\Phi(v)=f^*m_Yv \otimes t^i+c(v) \otimes it^{i-1}dt$$
  where $c(v)$ is chosen so that $dc(v)=(-1)^{k+1}f^*m_Yv+c(dv)$.  Here $c(dv)$
  is defined by induction so that
  $$\Phi(dv)=f^*m_Y(dv) \otimes t^i+c(dv) \otimes it^{i-1}dt;$$
  we know $\Phi(dv)$ takes this form by positive weights and the definition of
  $\Phi$ on lower-degree indecomposables.  Moreover, by the same argument as in
  the proof of Theorem \ref{thm:htpy}(ii), there is no obstruction to finding
  such a $c(v)$.

  The $t^i$-coefficients of $\Phi$ always have operator norm $L^k$; moreover, by
  Lemma \ref{lem:IP} and induction on $j$, the $t^{i-1}dt$-coefficients $c(v)$
  can be chosen so that for $v \in V_k \cap \mathcal{M}_j$, the operator norm is
  bounded by $C(X,Y)L^{k+j-1}$.  This gives us $\Dil_{C(X,Y)L^{d-1}} \Phi \leq L$;
  plugging this into Theorem \ref{tech:htpy}, with $\alpha=1$, gives the result.
\end{proof}

\subsection{Maps between spheres} \label{S:spheres}

The previous applications are to problems of great generality.  But the shadowing
principle can also be applied to yield new results in the much more specific
situation of maps between spheres, beyond the results of \cite{CDMW} and
\cite{CMW}.

\subsubsection*{Sharper bounds on nullhomotopies}

As noted before, the bounds of Theorem \ref{thm:posw} are sharp for certain
classes of examples.  On the other hand, they turn out not to be sharp for
example for maps $X \to S^n$, including $X=S^m$.  In that case, in the dimension
range where Hopf invariants play a role (when $n$ is even and $\dim X \geq 2n-1$)
Theorem \ref{thm:posw} only yields a quadratic bound on length.

In fact, this bound cannot be improved simply by choosing antidifferentials in a
clever way; it seems likely that any attempt to construct uniformly low-degree
polynomial nullhomotopies with a sharper bound would be similarly foiled.
Consider the map $f:S^3 \to S^2$ given by the connect sum of
$[L^2\id_{S^2},L^2\id_{S^2}]$ (on the northern hemisphere of $S^3$) and
$-[L^2\id_{S^2},L^2\id_{S^2}]$ (on the southern hemisphere).  The method of
constructing an algebraic nullhomotopy
$$\Phi:\mathcal{M}_{S^2}^*=\langle x^{(2)},y^{(3)} \mid dy=x^2 \rangle \to
\Omega^*(S^3) \otimes \mathbb{R}\langle t,dt \rangle$$
of $f$ used in the proof of Theorem \ref{thm:posw} yields
\begin{align*}
  \Phi(x) &= {f^*d\vol} \otimes t+\alpha \otimes dt \\
  \Phi(y) &= -\eta \otimes 2tdt
\end{align*}
where $d\alpha=-f^*d\vol$ and $d\eta=-{f^*d\vol} \wedge \alpha$.

Note that we can choose $\alpha$ to be zero on the equator; in that case, by
Stokes' theorem, every choice of $\eta$ must satisfy
$$\int_{S^2_{\text{eq}}} \eta=\int_{D^3_{\text{south}}} {f^*d\vol} \wedge \alpha=-L^4.$$
Indeed, consider any other choice $\hat\alpha=\alpha+\beta$ where $d\beta=0$.
Then $\beta=d\gamma$ for some $\gamma$, and therefore
$$\int_{D^3_{\text{south}}} {f^*d\vol} \wedge \hat\alpha-
\int_{D^3_{\text{south}}} {f^*d\vol} \wedge \alpha=
\int_{D^3_{\text{south}}} d({f^*d\vol} \wedge \gamma)=
\int_{S^2_{\text{eq}}} {f^*d\vol} \wedge \gamma=0,$$
since $f^*d\vol=0$ on the equator.  In other words, $\int_{S^2_{\text{eq}}} \eta$ does
not depend on our choices, and every nullhomotopy of this format must have formal
length $L^4$.  I would like to thank the referee for pointing this out.

One can do better by constructing and manipulating genuine, geometric
homotopies.\footnote{Of course, such homotopies have polynomial approximations.}
At the same time, the inductive approach given here forces the thickness of the
homotopy to grow.  I suspect that linear homotopies always exist in this
situation but that finding them will require new tools.  In fact, Sasha Berdnikov
has constructed such linear homotopies in the case $S^3 \to S^2$ using a purely
geometric method \cite{Berd}.
\begin{thm} \label{thm:weird}
  Let $n$ be even and $X$ a finite simplicial complex with $\dim X \geq 2n-1$.
  Then every $L$-Lipschitz nullhomotopic map $X \to S^n$ has an
  $O(L\exp(\kappa\sqrt{\log L}))$-Lipschitz nullhomotopy $X \times [0,1] \to S^n$
  for some constant $\kappa=\kappa(X,n)$.  In particular, this function is
  $o(L^{1+\epsi})$ for every $\epsi>0$.
\end{thm}
\begin{proof}
  Let $\gamma(L)$ be the best possible such function; we will show using a
  recurrence relation that $\gamma$ grows at most as fast as the above function.

  Let $f:X \to S^n$ be a nullhomotopic $L$-Lipschitz map.  The idea is to first
  nullhomotope a ``slightly shrunken copy'' of $f$.  We then expand this
  nullhomotopy again to get a nullhomotopy of $f$.  In a way, this is similar to
  Theorem \ref{thm:posw}; the key point is that in this case working with a
  genuine nullhomotopy lets us make this map not too much bigger.

  Define
  $$m_{S^n}:\mathcal{M}_{S^n}^*=\langle a^{(n)},b^{(2n-1)} \mid da=0,db=a^2 \rangle
  \to \Omega^*S^n$$
  via $m_{S^n}(a)=d\vol$ and $m_{S^n}(b)=0$, and let $C'=C'(m,S^n)$ be the constant
  given in the shadowing principle.  Let $\rho(L)$ be some function
  asymptotically below $L$.  We apply the principle to get a DGA homotopy
  $$\Phi:\mathcal{M}_{S^n}^* \to \Omega^*X \otimes \mathbb{R}\langle t,dt\rangle$$
  between the DGA map $\frac{1}{[C'\rho(L)]^n}f^*m_{S^n}$ and a nullhomotopic
  $L/\rho(L)$-Lipschitz map $g:X \to S^n$, such that
  $\Dil_{C'\rho(L)/L} \Phi \leq L/\rho(L)$.  This is possible since $f^*m_{S^n}$ is
  (algebraically) nullhomotopic, and hence so is its scaled version.  Finally, we
  choose a $\gamma(L/\rho(L))$-Lipschitz nullhomotopy $G:S^m \times [0,1] \to Y$
  of this $g$.

  Now we produce a new homomorphism
  $$\Psi:\mathcal{M}_{S^n}^* \to \Omega^*(X \times [0,2])$$
  satisfying $\Psi|_{t=0}=f^*m_{S^n}$ and $\Psi|_{t=2}=0$, as well as
  $\Dil\Psi \leq \max\{(C')^{\frac{2n}{2n-1}}L\rho(L),C'\rho(L)\gamma(L/\rho(L))\}$,
  by setting
  \begin{align*}
    \Psi(a)|_{(x,s)} &= \left\{\begin{array}{l l}
    {[C'\rho(L)]^n}\Phi(a)|_{t=s} & 0 \leq s \leq 1 \\
    {[C'\rho(L)]^n}G^*d\vol|_{(x,s-1)} & 1 \leq s \leq 2
    \end{array}\right. \\
    \Psi(b)|_{(x,s)} &= \left\{\begin{array}{l l}
    {[C'\rho(L)]^{2n}}\Phi(b)|_{t=s} & 0 \leq s \leq 1 \\
    0 & 1 \leq s \leq 2.
    \end{array}\right.
  \end{align*}
  The key is the observation that since $G^*m_{S^n}(b)=0$, the dilatation of
  $G^*m_{S^n}$ scales linearly as we expand.  Meanwhile, $\Phi$ was small to begin
  with and so the fact that it scales superlinearly doesn't matter very much.

  Finally, we apply Theorem \ref{tech:htpy} to get a $C(\Dil\Psi+1)$-Lipschitz
  nullhomotopy of $f$ on this interval.  Thus (assuming $L$ is large enough to
  ignore the additive constant) we obtain that
  \begin{equation} \label{recur}
    \gamma(L) \leq 2C\max\{(C')^{\frac{2n}{2n-1}}L\rho(L),
    C'\rho(L)\gamma(L/\rho(L))\}.
  \end{equation}
  Now choose $\kappa=\sqrt{2\log(2CC')}$ and $\rho(L)=\exp(\kappa\sqrt{\log L})$.
  Let $L_0$ be such that $\rho(L)<L$ for $L>L_0$ and fix a
  constant $A \geq 2C(C')^{\frac{2n}{2n-1}}$ such that for $1 \leq L \leq L_0$,
  $\gamma(L) \leq AL\rho(L)$.  Such a constant exists simply because we are
  maximizing over a bounded interval.  Now given $L>L_0$, suppose by induction
  that
  $$\gamma(L/\rho(L)) \leq A\frac{L}{\rho(L)}\rho(L/\rho(L)).$$
  Then \eqref{recur} implies that
  $$\gamma(L) \leq \max\left\{AL\rho(L), A \cdot 2CC'L
  \exp\left(\kappa\sqrt{\log L-\kappa\sqrt{\log L}}\right)\right\}.$$
  The term $\kappa\sqrt{\log L-\kappa\sqrt{\log L}}$ has a Taylor expansion
  $$\kappa\sqrt{\log L}-\frac{\kappa^2}{2}-\frac{\kappa^3}{8\sqrt{\log L}}
  -\cdots$$
  with all subsequent terms negative, and so we get that
  $\gamma(L) \leq AL\exp(\kappa\sqrt{\log L})$ as desired.
\end{proof}
Note that while this theorem yields eventual low growth, the number $L_0$ may be
extremely large; (arbitrarily) plugging in $\kappa=5$ yields an intersection
point $L=\rho(L)$ at $L \approx 7.2 \times 10^{10}$.  Before that point, the
theorem does not yield an estimate any better than the quadratic one.  It may be
possible to obtain better estimates in this low-$L$ range by using the same
method with $\rho(L)=L^\epsi$, for various fixed $\epsi>0$, to show directly that
$\gamma(L)=o(L^{1+\epsi})$.

The proof above uses only the following facts about $S^n$:
\begin{itemize}
\item $S^n$ is geometrically formal (or more generally, the map
  $\Omega^*S^n \to H^*(S^n;\mathbb{R})$ admits a splitting algebra homomorphism);
\item $S^n$ admits automorphisms which multiply elements of $H^k$ by $t^k$, for
  some $t$.
\end{itemize}
In fact, the second property is a consequence of formality \cite{Shiga}.  The
same proof (with slight modifications depending on the height of the positive
weight filtration defined in Theorem \ref{thm:posw}) provides a bound of the form
$O(L\exp(\kappa(X,Y)\sqrt{\log L}))$ for nullhomotopies any map from a finite
complex $X$ to a $Y$ for which $\Omega^*S^n \to H^*(S^n;\mathbb{R})$ splits.  For
example, this gives such a bound for nullhomotopies of maps from $S^m$ to a wedge
of $n$-spheres and of maps to symmetric spaces, or more generally to wedges of
symmetric spaces.  This completes part (iii) of Theorem \ref{summary:sym}.

\subsubsection*{Uniformity over the metric}

In \cite{Guth}, Larry Guth asks the following question:
\begin{ques}
  The $n$-dimensional ellipse with principal axes $R_0,\ldots,R_n$ is the set
  defined by
  $$\sum_{j=0}^n (x_j/R_j)^2=1.$$
  Let $E^m$ and $F^n$ be $m$- and $n$-dimensional ellipses, respectively.  If
  $f:E \to F$ is nullhomotopic and $L$-Lipschitz, can we homotope $f$ to a
  constant map through maps of Lipschitz constant at most
  $L^\prime=L^\prime(m,n,L)$, independent of the dimensions of $E$ and $F$?  Can
  this be taken to be $C(m,n)L$ or $C(m,n)L^2$ as dictated by the rational
  homotopy?  What about more complicated metrics on the sphere?
\end{ques}
Such ellipses, and any metric on the sphere, can be closely approximated by
simplicial complexes after sufficient scaling.  Therefore Theorem \ref{thm:posw}
allows us to give a half-answer to this which is, however, less than half
satisfying.  As long as we fix the target metric on the sphere and $L$ is larger
than some constant depending on the domain metric\footnote{Roughly the inverse of
  the mesh size.}, the nullhomotopy can go through maps of Lipschitz constant at
most $C(F,m,n)L$ (if $n$ is odd or $m<2n-1$) or $C(F,m,n)L^2$ (otherwise.)
However, dependence on the target metric is a complete mystery, as it is in this
whole paper.

Recent results in \cite{FFWZ} suggest that the constants depending on the target
space can be quite large even for relatively small target spaces in fixed
dimension.  On the other hand, it may be that things are less dire when we
restrict to spaces of the same homeomorphism type.

\bibliographystyle{amsalpha}
\bibliography{liphom}
\end{document}